\documentclass[11pt,reqno]{amsart}

\usepackage{hyperref}
\usepackage{graphicx}
\usepackage{color}
\usepackage{amsfonts,amssymb}
\usepackage{bbm} 
\usepackage{psfrag}
\usepackage{mathrsfs}
\usepackage{pdfsync}
\usepackage{a4wide}

\usepackage{mathabx}

\usepackage{tikz}

\newtheorem{neu}{}[section]
\newtheorem{Cor}[neu]{Corollary}
\newtheorem*{Cor*}{Corollary}
\newtheorem{Thm}[neu]{Theorem}
\newtheorem*{Thm*}{Theorem}
\newtheorem{Prop}[neu]{Proposition}
\newtheorem*{Prop*}{Proposition}
\theoremstyle{definition}
\newtheorem{Lemma}[neu]{Lemma}
\newtheorem*{Rmk*}{Remark}
\newtheorem{Rmk}[neu]{Remark}

\newtheorem*{Ex*}{Example}

\newtheorem*{Qu*}{Question}

\newtheorem{Def}[neu]{Definition}

\newcommand{\N}{\mathbb{N}}

\newcommand{\Z}{\mathbb{Z}}
\newcommand{\R}{\mathbb{R}}
\newcommand{\C}{\mathbb{C}}
\newcommand{\CP}{\C\mathrm{P}}

\newcommand{\om}{\omega}
\newcommand{\Om}{\Omega}

\newcommand{\A}{\mathcal{A}}

\newcommand{\M}{\mathcal{M}}
\newcommand{\Ma}{\mathfrak{M}}

\renewcommand{\H}{\mathrm{H}}

\newcommand{\CM}{\mathrm{CM}}

\newcommand{\Crit}{\mathrm{Crit}}

\newcommand{\CL}{\operatorname{cuplength}}

\newcommand{\beq}{\begin{equation}}
\newcommand{\beqn}{\begin{equation}\nonumber}
\newcommand{\eeq}{\end{equation}}

\newcommand{\bea}{\begin{equation}\begin{aligned}}
\newcommand{\bean}{\begin{equation}\begin{aligned}\nonumber}
\newcommand{\eea}{\end{aligned}\end{equation}}

\usepackage{marginnote}

\numberwithin{equation}{section}

\definecolor{blue}{rgb}{0,0,1}
\definecolor{red}{rgb}{1,0,0}

\begin{document}
\title{Periodic Reeb orbits on prequantization bundles}
\author{Peter Albers \and Jean Gutt \and Doris Hein}
\address{
    Peter Albers\\
   Mathematisches Institut\\
    Ruprecht-Karls-Universit\"at Heidelberg}
\email{palbers@mathi.uni-heidelberg.de}
\address{Jean Gutt\\
Mathematisches Institut\\
Universit\"at zu K\"oln}
\email{gutt@math.uni-koeln.de}
\address{Doris Hein\\
Mathematisches Institut\\
Albert-Ludwigs-Universit\"at Freiburg}
\email{doris.hein@math.uni-freiburg.de}
\keywords{}
\subjclass[2000]{}
\begin{abstract}
	In this paper, we prove that every graphical hypersurface in a prequantization bundle over a symplectic manifold $M$, pinched between two circle bundles whose ratio of radii is less than $\sqrt{2}$ carries either one short simple periodic orbit or carries at least $\operatorname{cuplength}(M)+1$ simple periodic Reeb orbits.
\end{abstract}
\maketitle

\section{Introduction}
A \emph{contact form} on  a manifold $M$ of dimension $2n-1$ is a differential $1$-form $\alpha$ satisfying $\alpha \wedge (d\alpha)^{n-1} \neq 0$ everywhere. The \emph{Reeb vector field} $R_\alpha$ associated to a contact form $\alpha$ is the unique vector field on $M$ characterized by: $\iota(R_\alpha)d\alpha=0$ and $\alpha(R_\alpha)=1$.

In this article we are concerned with prequantization bundles $E$. That is, $E$ is a $\C$-bundle over a symplectic manifold $(M,\om)$ with $c_1(E)=-[\om]\in\H^2(M;\Z)$. In particular, we assume that the cohomology class $[\om]$ of the symplectic form admits an integral lift. A Hermitian connection on $E$ gives rise to a connection 1-form $\alpha$ on the corresponding $S^1$-bundle $\Sigma$ over $M$. The 1-form $\alpha$ is naturally a contact form. Its Reeb vector field is the infinitesimal generator of the $S^1$-action on $\Sigma$, see \cite[Section 7.2]{Geiges_book} for more details. Moreover, the Hermitian connection defines circle resp.~disk bundles $S_R$ resp.~$D_R$ of radius $R>0$. We will extend $\alpha$ to $E\setminus M$ by pullback.

We call a hypersurface $\Sigma_f\subset E$ \emph{graphical} if it can be written as the graph of a function $f:\Sigma\to \R_{>0}$ inside $E$
\begin{equation}
\Sigma_f=\{f(x)x\mid x\in \Sigma\}\;.
\end{equation}
We call $\Sigma_f$ \emph{pinched} between $S_{R_1}$ and $S_{R_2}$ if $\Sigma\subset D_{R_2}\setminus \mathrm{int}D_{R_1}$.
Then $\alpha_f:=f\alpha$ is a contact form on $\Sigma$. The extension of $\alpha$ to $E\setminus M$ is again a contact form if restricted to $\Sigma_f$ and the Reeb flows of $\alpha$ on $\Sigma_f$ and of $\alpha_f$ on $\Sigma$ are equivalent by radial projection.

\begin{Thm}
\label{Main Theorem}
Let  $E$ be prequantization bundle over the closed symplectic manifold $(M^{2n},\om)$. Assume that the graphical hypersurface $\Sigma_f\subset E$ is pinched between  $S_{R_1}$ and $S_{R_2}$ with $\frac{R_2}{R_1}<\sqrt{2}$. Then there exist either infinitely many simple periodic Reeb orbits of $R_{\alpha_f}$ on $\Sigma$ or there are periodic orbits $\gamma_1,\ldots,\gamma_{c}$ of $R_{\alpha_f}$ with $c=\operatorname{cuplength}(M;\Z/2)+1$ such that
	\[
		\pi R_1^2<\A_{\alpha_f}(\gamma_1)<\hdots<\A_{\alpha_f}(\gamma_{c})<\pi R_2^2
	\]
where $\A_{\alpha_f}(\gamma):=\int_{\gamma}\alpha_f$ is the action or period of a Reeb orbit $\gamma$.
\end{Thm}
We recall the definition of cuplength.
\begin{Def}
	Let $M$ be a manifold.
	The \emph{cuplength} of $M$ (with coefficients in $\Z/2$) is defined as
	\[
		\operatorname{cuplength}(M;\Z/2) := \max\left\{k\in \N\,|\,\exists \beta_1,\ldots,\beta_k\in H^{\geq1}(M;\Z/2) \textrm{ such that } \beta_1\cup\hdots\cup\beta_k\neq0\right\}.
	\]
\end{Def}
\begin{Cor}\label{cor:multiplicity}
In the context of Theorem \ref{Main Theorem}, either the minimal period of periodic Reeb orbits of $R_{\alpha_f}$ is less than $\pi R_1^2$ or $\alpha_f$ carries at least $\operatorname{cuplength}(M)+1$ simple periodic Reeb orbits. 
\end{Cor}
  
In short, there is either a short periodic orbit on a pinched graphical hypersurface $\Sigma_f$ or $\operatorname{cuplength}(M)+1$ simple periodic Reeb orbits. 
\begin{Rmk}
Amongst other theorems, a similar result to Theorem \ref{Main Theorem} as been obtained by Ely Kerman in \cite{Kerman_rigid_constellations_of_closed_Reeb_orbits}, but with the bound for the number of critical points being $\tfrac12\dim M+1$. Since the symplectic form is non-degenerate, we have $\operatorname{cuplength}(M)+1\geq\tfrac12\dim M+1$ for closed symplectic manifolds from Stokes' theorem. 
\end{Rmk}

As a particular case of Corollary \ref{cor:multiplicity} together with an observation by \cite{BLMR}, we also find the following. We recall that $S^{2n-1}$ is the $S^1$-bundle corresponding to a prequantization bundle over $\CP^{n-1}$ and $\operatorname{cuplength}(\CP^{n-1})=n-1$. 
	
\begin{Cor}[\cite{EL, BLMR}]\label{thmintro:el}
	Let $\Sigma$ be a hypersurface in $\R^{2n}$ satisfying
\begin{equation}
\langle \nu_{\Sigma}(x),x\rangle >r \quad\forall x \in \Sigma, 
\end{equation}	
where $\nu_{\Sigma}(x)$ is the exterior unit normal vector of $\Sigma$ at $x$. Then $\Sigma$ is starshaped and we denote by $\xi = \ker \alpha_0$ the standard contact structure on $\Sigma$.
Assume there exists a point $x_{0} \in \R^{2n}$ and numbers $0<r\leq R$ with $R<r\sqrt{2}$ such that:
	\begin{equation}
		r\leq \Vert x-x_{0}\Vert\leq R \quad\forall x \in \Sigma , \quad 
	\end{equation}
Then $\Sigma$ carries at least $n$ geometrically distinct periodic Reeb orbits.
\end{Cor}
Another proof of this result with the additional assumption that the contact form is non degenerate was given by the second author in \cite{G}.

The study of periodic Reeb orbits can be translated in the study of periodic solutions of Hamiltonian systems and has a long history which, probably, started when Poincar\'e pointed out their interest.
The question of lower bounds on the number of simple periodic Reeb orbits on compact manifold is wide open; it is not even known for the standard contact structure on the sphere in $\R^{2n}$.
In fact, the existence of one periodic Reeb orbit on every compact contact manifold (Weinstein conjecture) is still open in dimension greater than 3 where it was proven by Taubes \cite{Taubes_Weinstein_conjecture_I}.
Taubes result was then improved independently by Cristofaro-Gardiner and Hutchings \cite{GH} and by Ginzburg, Hein, Hryniewicz and Macarini \cite{GHHM}, who proved that
every contact form on a closed three-manifold has at least two embedded periodic Reeb orbits.

On the sphere, more is known; Hofer, Wysocki and Zehnder \cite{HWZ} have shown that on $S^{3}$, every dynamically convex (see \cite{HWZ}) contact form carries either 2 or infinitely many periodic Reeb orbits.
In dimension greater than $3$, the conjecture is that any contact form on the $2n-1$ dimensional sphere defining the standard contact structure admits at least $n$ simple periodic orbits.
This conjecture is studied, for instance, in \cite{GG2,LZ,WHL,EL,BLMR}

For manifolds (of dimension $\geq5$) other than the sphere, very little is known, we refer to \cite{GG2,GK, AM,Kang_Equivariant_symplectic_homology_and_multiple_closed_Reeb_orbits} for precise statements but we would like to point out that nothing is known outside some restricted class of prequantization bundles. For prequantization bundles Ginzburg proved an analog of our main theorem in the $C^0$-small case, see \cite[Theorem 2.7]{Ginzburg_96}.


\subsection*{Acknowledgements}
The authors are grateful to W. Merry for very useful discussions. The authors warmly thanks the anonymous referee for his careful reading and many improvement suggestions.
P.A and J.G are partially supported by the SFB/TRR 191 ``Symplectic Structures in Geometry, Algebra and Dynamics'', funded by the Deutsche Forschungsgemeinschaft.

\section{Basic constructions}

Let $(\Sigma,\alpha)$ be a prequantization space over $(M,\om)$. That is, $(M,\om)$ is a closed connected symplectic manifold with integral symplectic form $[\om]\in\H^2(M,\Z)$. We denote by $\wp:\Sigma\to M$ the principal $S^1$-bundle and by $\wp:E\to M$ the associated complex line bundle with first Chern class $c_1^E=-[\om]$. We refer to these bundles as prequantizations spaces. There exists an $S^1$-invariant 1-form $\alpha$ on $\Sigma$, and hence $E\setminus M$, with the property
\beq
d\alpha=\wp^\star\om
\eeq
which is a contact form on $\Sigma$. For more details we refer to \cite[Section 7.2]{Geiges_book}. If we denote by $\rho$ the radial coordinate on $E$ then the 2-form
\beq
\Om:=d\big(\pi \rho^2\alpha\big)+\wp^\star\om=2\pi \rho d \rho\wedge\alpha+\big(\pi \rho^2+1\big)\wp^\star\om
\eeq
is a symplectic form on $E$.

In the following we will work on the symplectization $S\Sigma:=\Sigma\times \R_{>0}$ of $\Sigma$ which is equipped with the exact symplectic form $\Om=d(r\alpha)=dr\wedge \alpha+rd\alpha$. Here $r$ is the natural coordinate on $\R_{>0}$. The coordinate transformation $r=\pi\rho^2$ induces an exact symplectomorphism $(E\setminus M,\pi \rho^2\alpha)\cong(S\Sigma,r\alpha)$. We point out that the Reeb flow $\theta_t$ of the Reeb vector field $R$ on $\Sigma$ is 1-periodic due to our convention that $S^1=\R/\Z$.

Note that the question of number of periodic Reeb orbits is invariant by rescaling.
Therefore, in the following, we shall take $\pi R_1^2=1$ and thus, we know $\pi R_2^2<2$; we shall denote $\pi R_2^2$ by $R_0$.

\subsection{The Hamiltonian functions and their periodic orbits}
In this paper, the initial choice of the Hamiltonian function plays a crucial role. It is defined as a radial function in the complex line bundle $E\to M$ and has a shape similar to the standard ones in symplectic homology, but eventually becoming constant again.

In order to construct this Hamiltonian function, we first fix a number $R_0\in\R$ with $1<R_0<2$ and choose constants $A,c\in(0,1)$, which in addition satisfy
\begin{equation}
\label{eqn:choices}
\begin{aligned}
&c<\frac{R_0-1}{1-\log R_0},\\[1ex]
&Ac\big(\exp{\tfrac{R_0-1}{c}}-1\big)<1.
\end{aligned}
\end{equation}
The first condition is only needed if $R_0$ is close to 1 as otherwise, the right hand side of the first equation in  \eqref{eqn:choices} is larger than 1 and automatically satisfied by $c\in(0,1)$. The second condition can then be satisfied by choosing $A$ sufficiently small.
 Then we define the function $k\colon\R_{>0}\to\R$ explicitly by the formula
\begin{equation}
k(r)=cr\log r-cr+r(1-c\log A)+Ac-A\;.
\end{equation}
Therefore, we have $k(A)=0, k'(A)=1$ and 
\beq
\label{bound Hessian}
|rk''(r)|=c<1.
\eeq
We next set $B=A\exp\frac{R_0-1}{c}$. Thus, $A< B\approx Ac$, where $\approx$ becomes an equality in the limit $c\to 1$ and $R_0\to 2$. Moreover, the relations between $A$, $c$ and $R_0$ in \eqref{eqn:choices} are equivalent to the more readable conditions 
\begin{equation}
\log A+1<\log R_0B
\end{equation}
and
\beq\label{small_value}
c(B-A)<1.
\eeq
To define $h\colon\R_{\geq 0}\to \R$, we fix sufficiently small constants $\epsilon,\delta,\bar\delta>0$ and set
\beq
h(r)=k(r)\quad\text{for } r\in[A-\bar\delta,B+\delta]\
\eeq
and require 
\begin{equation}
h'(B+\delta)=R_0+\epsilon\;.
\end{equation}
For $r\leq A-\bar\delta$, we choose $h$ to be almost linear down to $r=\bar\delta$  and then turning to be constant such that 
$
-h(0)\notin\left[A,A+c(B-A)\right]+\Z.
$
This can be achieved by making $\bar\delta$ sufficiently small and keeping the property \eqref{bound Hessian}.
For $r\geq B+\delta$, we choose $h$ to be constant with slope $R_0+\epsilon <2$ for some time and then decrease the slope to 
\begin{equation}
h'(C)=R_0\quad \text{at some point } r=C>B.
\end{equation}
After this, we keep slope $R_0-\epsilon$ for a while until we decrease again to $h'(D)=1$ for some possibly large $D>C$. By the same pattern, we decrease the slope further to $1-\epsilon$ for some finite interval before we eventually make $h(r)$ constant for large $r$.
In the non-linear parts, we make all choices such that the condition \eqref{bound Hessian} is satisfied, i.e., the slope decreases more slowly as we move further out.

We adjust the various bits of constant slope, $R_0\pm\epsilon,1-\epsilon$, so that the respective values of $h$ at $B,C,D$ are such that the requirements below are met. To sum this up, we construct $h(r)$ such that we get a shape as in Figure \ref{pic:h} with the following properties:
\begin{equation}\label{eqn:properties_of_h}
\begin{cases}
h'(r)\in[0,R_0+\epsilon] &\text{for all }r\in \R_{>0}\\[1ex]
\displaystyle\max h'(r)=R_0+\epsilon<2\\[1ex]
h'(r)=0 & \Longleftrightarrow\ r \in[0,\bar\delta]\text{ or } r \text{ large} \\[1ex]
h'(r)=1 &\Longleftrightarrow\; r\in\{A,D\}\\[1ex]
h'(r)=R_0 &\Longleftrightarrow\; r\in\{B,C\}\\[1ex]
h(A)=0,\\[1ex]
-h(0)\notin \left[A,A+c(B-A)\right]+\Z\\[1ex]
h(B)=BR_0-cB+cA-A\\[1ex]
CR_0-h(C)\notin \left[A,A+c(B-A)\right]+\Z\\[1ex]
D-h(D)\notin \left[A,A+c(B-A)\right]+\Z\\[1ex]
\displaystyle \lim_{r\to\infty}h(r)\notin \left[A,A+c(B-A)\right]+\Z\\[1ex]
h''(r)\geq0 &\text{for } r\leq R_0B\\[1ex]
|rh''(r)|<1&\text{for all }r\in \R_{>0}\\[1ex]
\end{cases}
\end{equation}
Note that by our choices of $R_0, c$ and $A$, the set $\left[A,A+c(B-A)\right]+\Z$ is not all of $\R$ as we required $c(B-A)<1$ and the conditions on $h(0), C,D$ and $\lim_{r\to\infty}h(r)$ can be satisfied.

Most of these conditions are needed to get a good picture of periodic orbits and their action values. But we point out that the last condition is the most important one as it will enable us to get hold of a certain moduli space (in Step 2 of Theorem \ref{thm:moduli space}).
\begin{figure}[ht] 
\def\svgwidth{80ex} 
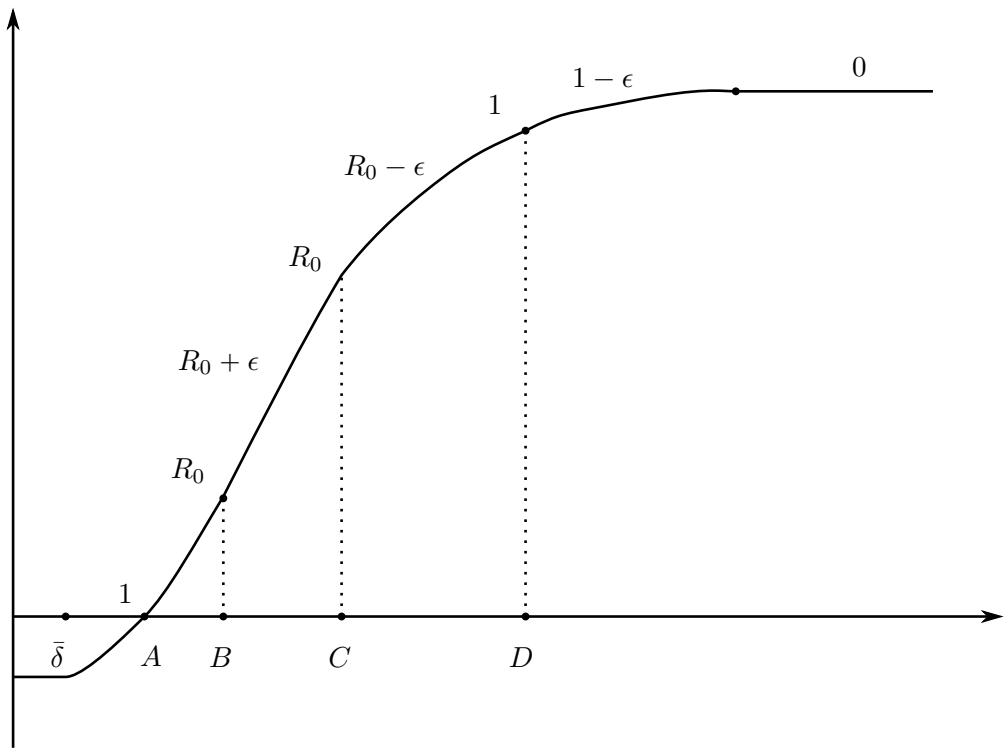
\caption{The function $h$. The numbers at the graph indicate the slope at this point / section.}\label{pic:h}
\end{figure}
Using this function $h(r)$, we now define the Hamiltonian function $H\colon E\to\R$ simply by 
\begin{equation}
H(q):=\begin{cases}
h(r) & \text{if }q=(x,r)\in E\setminus M\cong\Sigma\times\R_{>0} \\
h(0) & \text{if }q\in M
\end{cases}
\end{equation}
The Hamiltonian function $H$ is smooth since $h(r)$ is constant for $r<\bar\delta$.

As a next step, we compute the action values of all 1-periodic orbits for this Hamiltonian function $H=h(r)$. Observe that with our conventions, the Hamiltonian vector field is given by $X_H=h'(r)R$ where $R$ is the Reeb vector field on $\Sigma$. Moreover, since the Reeb flow is 1-periodic the 1-periodic orbits of $X_H$ correspond to values of $r$ with $h'(r)\in\Z$. As we chose $R_0+\epsilon<2$ to be the maximal slope of $h$, the condition $h'\in\Z$ for 1-periodic orbits turns into $h'(r)\in\{0,1\}$. We get four types of periodic orbits:
\begin{enumerate}
\item Constant orbits for $r\in [0,\bar\delta]$, where $h$ is constant,
\item 1-periodic Reeb orbits at $r=A$, where $h'(A)=1$,
\item 1-periodic Reeb orbits at $r=D$, where $h'(D)=1$ and
\item Constant  orbits for very large $r$, where $h$ is again constant. 
\end{enumerate}
To compute the action values of these periodic orbits we recall that the Hamiltonian  action functional $\A_H$ on $(E,\Om)$ is defined on a covering $\tilde\Lambda E$ of the component of contractible loops of the free loop space $\Lambda E$ of $E$. This covering has $\frac{\pi_2(E)}{\ker \Om}\cong\frac{\pi_2(M)}{\ker \om}$ as deck transformation group. We denote elements by $[\gamma,\bar\gamma]$; i.e.~concretely $\gamma$ is a contractible loop and $\bar\gamma$ is a disk bounded by $\gamma$ with the equivalence relation that $(\gamma,\bar\gamma)\sim(\gamma',\bar\gamma')$ if and only if
\begin{equation}
\begin{cases}
\gamma=\gamma' & \\
\Om(\bar\gamma\#-\bar\gamma')=0\;. & \\
\end{cases}
\end{equation}
We call $\bar\gamma$ a capping of $\gamma$. The action functional $\A_H:\tilde\Lambda E\to \R$ is defined by
\begin{equation}
\A_H([\gamma,\bar\gamma]):=\int_{D^2}\bar\gamma^*\Om-\int_0^1H\big(\gamma(t)\big)dt\;.
\end{equation}
An element $[\gamma,\bar\gamma]$ is a critical point of $\A_H$ if and only if
\begin{equation}
\gamma'(t)=X_H\big(\gamma(t)\big)
\end{equation}
i.e. the 1-periodic orbits of $X_H$ with some capping $\bar\gamma$. 

As explained above there are 4 types of orbits, either constant orbits or 1-periodic Reeb orbits, each at certain values for $r$. We point out that all these orbits have natural cappings. For the constant orbits we choose the capping to be a constant disk. For the 1-periodic Reeb orbits we choose the disk in a fiber of $E$ containing the specific Reeb orbit. Using these natural cappings we abbreviate their action values by $\A_H(r)$. Then a simple computation (we recall our convention $S^1=\R/\Z$) leads to
\begin{equation}
\A_H(r)=rh'(r)-h(r)\;.
\end{equation}
The action values of the critical points with other cappings are obtained by changing the natural cappings by an element in $\frac{\pi_2(E)}{\ker \Om}\cong\frac{\pi_2(M)}{\ker \om}$. This changes the action value by an integer since $\om:\frac{\pi_2(M)}{\ker \om}\to \Z$ due to the condition that $[\om]\in \H^2(M;\Z)$.

We now compute the action values $\A_H(r)$ for the orbits of different types using the properties of $h$, see \eqref{eqn:properties_of_h}.
\begin{itemize}
\item For the orbits in class (1), our choice of $h$ implies $\A_H(r)=-h(0)\notin \left[A,A+c(B-A)\right]+\Z$ for $r\in[0,\bar\delta]$.
\item For orbits in class (2), we fixed the value of the Hamiltonian function to be zero and therefore get $\A_H(A)=A$.
\item For the orbits in class (3), we required $\A_H(D)=Dh'(D)-h(D)\notin \left[A,A+c(B-A)\right]+\Z$.
\item Finally, for class (4), the condition on $\lim_{r\to\infty} h(r)$ and that $h$ becomes constant for large $r$ imply that the action value $\A_H(r)$ is not in $\left[A,A+c(B-A)\right]+\Z$.
\end{itemize}
For the second Hamiltonian function $L$, we consider a rescaled version of $h$ by defining
\begin{equation}
l(r)=h\left(\frac{r}{R_0}\right).
\end{equation}
and define as above
$L:E\to\R$  by 
\begin{equation}
L(q):=\begin{cases}
l(r) & \text{if }q=(x,r)\in E\setminus M\cong\Sigma\times\R_{>0} \\
l(0) & \text{if }q\in M
\end{cases}
\end{equation}
In this case, we have 
\begin{equation}\label{Ham_vfield_for_L}
X_L(x,r)=\tfrac{1}{R_0}h'\big(\tfrac{r}{R_0}\big) R(x)
\end{equation} 
and therefore, to get a periodic orbit, we must have $h'\big(\tfrac{r}{R_0}\big)$ to be an integer multiple of $R_0$.
By the conditions on $h$, we still get the constant periodic orbits as above for $r\in[0,R_0\bar\delta]$ and for large $r$ and 1-periodic orbits when $r=R_0B$ and $r=R_0C$.

The action values with the natural cappings are now given by
\begin{equation}
\A_L(r)=rl'(r)-l(r)=\frac{r}{R_0}h'\left(\frac{r}{R_0}\right)-h\left(\frac{r}{R_0}\right). 
\end{equation}
Again the action values for the constant orbits are given by the values of $h$ near $0$ and near $\infty$ and therefore are not in $\left[A,A+c(B-A)\right]+\Z$.
For the orbits at $r=R_0B$, we find
\begin{equation}
\A_L(R_0B)=Bh'(B)-h(B)=BR_0-h(B)=A+c(B-A).
\end{equation}
Finally, at $r=R_0C$, the properties of $h$ imply that 
\begin{equation}
\A_L(R_0C)=Ch'(C)-h(C)\notin \left[A,A+c(B-A)\right]+\Z.
\end{equation}
From now on, we restrict our attention to periodic orbits with action in the interval $I=[A,A+c(B-A)]$. These are the orbits at the first time the slope reaches 1, when the Hamiltonian function starts to increase as usually considered in symplectic homology.
More concretely, we only have the orbits at $r=A$ for $H$ and at $r=R_0B$ for $L$ with their natural capping being the fiber disk in the complex line bundle $E$.

\subsection{The initial moduli space}\label{sec:initial_moduli_space}

We now study the moduli space arising in the continuation homomorphism for a monotone homotopy between the  Hamiltonian functions $H$ and $L$ coming from a monotone homotopy between $h(r)$ and $l(r)=h(\frac{r}{R_0})$ constructed above. For this we define 
\begin{equation}\label{eqn:def_of_h_s}
h_s(r)=\beta(s)h(r)+\big(1-\beta(s)\big)h\big(\tfrac{r}{R_0}\big),
\end{equation}
where $\beta$ is a smooth, monotone decreasing cut-off function which is 1 for $s<-1$ and 0 for $s>0$. Moreover, we require that $\beta'(s)<0$ for all $s\in (-1,0)$.
Then $h_s$ is a monotone homotopy from $h$ to $l$.
Note that the condition $|rh_s''(r)|<1$ is still satisfied for all $r>0$ and $s\in\R$ as for each $s$ the function $h_s$ is a convex combination of $h$ and $l$ which both satisfy the required condition.

Using the contact form $\alpha$ we can split the tangent space of $E$ into a horizontal subspace $H$ and a vertical subspace $V$, i.e.~$TE=V\oplus H$, i.e.~$V=\ker d\wp$ and $H=\ker \alpha$. In particular, $V$ is spanned by the radial vector field and the Reeb vector field. An almost complex structure $J$ on $E$ is therefore represented by a $2\times 2$-matrix in this splitting
\begin{equation}
J=\begin{pmatrix}
i & \mathfrak{B}\\
\mathfrak{A} & j
\end{pmatrix}:V\oplus H\to V\oplus H
\end{equation}
Here $i$ resp.~$j$ is an almost complex structure on $V$ resp.~$H\cong TM$ and $\mathfrak{B}:E\to \text{End}(H,V)$ and $\mathfrak{A}:E\to \text{End}(V,H)$ are maps from $E$ into the endomorphisms satisfying
\begin{equation}
i\mathfrak{B}+\mathfrak{B}j=0\quad\text{and}\quad j\mathfrak{A}+\mathfrak{A}i=0\;.
\end{equation}

\begin{Def}
\label{defn:SFT_type}
We fix an almost complex structure 
\begin{equation}
J=\begin{pmatrix}
i & 0\\
0 & j
\end{pmatrix}
\end{equation}
on $E$ such that $J\frac{\partial}{\partial r}\equiv i\frac{\partial}{\partial r}=\frac{1}{r}R$ and $j$ is $\om$-compatible via the identification $\wp:H\cong TM$.
\end{Def}
%
%
As we will use this very explicitly in the technical parts of this paper, we mention here that this implies that $\frac1r dr = \alpha\circ J$ and therefore $E$ is convex at infinity. Moreover, $\wp:E\to M$ is $J$-$j$-holomorphic and $J$ is $\Om$-compatible. 
 Using this $J$, we can now define the moduli space of interest as
\beq\label{eqn:moduli_space_continuation_traj}
\M=\left\{u\colon S^1\times\R\to E\;\Big|\;\partial_su+J(u)\big(\partial_tu-X_s(u)\big)=0, \;\;
\begin{aligned}
&\lim_{ s\to\pm\infty}u= \gamma_\pm\\[0.5ex] &\Om(\gamma_-\#u\#\bar\gamma_+)=0
\end{aligned}\;
 \right\},
\eeq
where $X_s$ is the Hamiltonian vector field for $h_s$ and by $\gamma_-\#u\#\bar\gamma_+$ we mean the sphere obtained by capping $u$ off with the fiber disks of $\gamma_-$ at $-\infty$ and the fiber disks with its orientation reversed of $\gamma_+$ at $+\infty$. Moreover, the orbit $\gamma_-$ is a 1-periodic orbit of $X_H$ at $r=A$ and $\gamma_+$ is a 1-periodic orbit of $X_L$ at $r=R_0B$, both having action in the interval $I=[A,c(B-A)-A]$ as computed above. Both Hamiltonian action functionals for $H$ and $L$ are Morse-Bott and the critical manifolds formed by the respective orbits $\gamma_\pm$ are both diffeomorphic to $\Sigma$ since $\gamma_\pm$ correspond to simple Reeb orbits.

We point out that since all Hamiltonian functions are autonomous, the moduli space $\M$ carries a free $S^1$-action given by rotating solutions, $(\tau\ast u)(s,t):=u(s,t+\tau)$, $\tau\in S^1$. That this action is free follows from considering the asymptotic limits of $u$. The main result of this section is the following.

\begin{Thm}
\label{thm:moduli space}
The space $\M$ of solutions $u$ to the Floer equation
\beq
\label{eqn:s_dep_Floer_eqn} 
\partial_su+J(u)\big(\partial_tu-X_s(u)\big)=0
\eeq
with 
\begin{equation}\label{eqn:moduli_problem}
\begin{aligned}
u(+\infty)&\in \big\{\gamma_+\in \Crit\A_L\mid\A_L(\gamma_+)=c(B-A)+A\big\}\\
u(-\infty)&\in \big\{\gamma_-\in \Crit\A_H\mid\A_H(\gamma_-)=A\big\}
\end{aligned}
\end{equation}
and 
\begin{equation}
\Om(\gamma_-\#u\#\bar\gamma_+)=0
\end{equation}
is compact and carries a free $S^1$-action. Moreover, it is $S^1$-equivariantly diffeomorphic to $\Sigma$
\begin{equation}
\M\cong_{S^1} \Sigma
\end{equation}
and thus
\begin{equation}
\M/S^1\cong M\;.
\end{equation}
Finally, all solutions $u\in\M$ are Fredholm regular.
\end{Thm}

Of course, the statement that $\M$ is compact follows from the rest of the statement as $\Sigma$ is compact. Therefore, we do not need to prove compactness separately and it suffices to prove $\M\cong \Sigma$. 

Even though we will not use it, it is worth pointing out that elements in $\M$ are contributions to the continuation homomorphism between the Floer homologies of $H$ and $L$.

Before proving this theorem, we give an outline of the proof by mentioning the main steps:
\begin{enumerate}
\item[Step 1] We first show that all elements in $\M$ are contained in a fiber over a periodic Reeb orbit $\gamma$ on $\Sigma$ of the bundle $E\to M$ by an energy estimate. In particular, this shows that a solution to the Floer equation in $\M$ can only exists if the asymptotic critical points $\gamma_\pm$ are in the same fiber, i.e., they correspond to the same Reeb orbit $\gamma$ on $\Sigma$.
\item[Step 2] According to Step 1 we write $u(s,t)=\big(\gamma\big(b(s,t)\big),F(s,t)\big)\in\M$, where $F$ is the radial coordinate. Then we show that all solutions $u\in\M$ satisfy $b(s,t)=t$ and $F(s,t)=F(s)$ for some function $F\colon\R\to\R_{>0}$, i.e.~$u(s,t)=\big(\gamma(t),F(s)\big)$.
\item[Step 3] According to Step 2 the Floer equation for $u$ reduces to an ODE for $F$. We prove existence and uniqueness of a solution $F$ for any fixed Reeb orbit $\gamma$ using the asymptotic conditions at both ends. This completes the proof of $\M\cong_{S^1}\Sigma$ and the equality of the $S^1$-actions by rotation in the fiber.
\item[Step 4]Finally, we prove Fredholm regularity for our solutions.
\end{enumerate}

\begin{proof}
As outlined above, the proof is done in several steps.\\

\noindent {\bf Step 1:} Let $u$ be an element of the moduli space $\M$ and define $v:=\wp (u)$, where $\wp\colon E\to M$ is the bundle projection. By our choice of $J$, see Definition \ref{defn:SFT_type}, the projection $\wp$ is holomorphic with respect to the complex structures $J$ on $E$ and $j$ on $M$. As the Reeb vector field (and thus also the Hamiltonian vector field $X_s$) always point in fiber direction, we have $\wp_\star X_s=0$. Therefore, $v$ solves the unperturbed Cauchy-Riemann equation
\begin{equation}
\partial_sv+j(v)\partial_tv=0
\end{equation}
and is of finite energy. By removal of singularity $v$ extends to a holomorphic sphere, which we denote by $v$ again. Since the asymptotic conditions for $u$ are two periodic orbits in a fiber of $E$ they project via $\wp$ to points in $M$ and thus $\wp(\gamma_-\#u\#\bar\gamma_+)=v$. The projection $\wp$ induces an isomorphism $\wp_*:\pi_2(E)\cong\pi_2(M)$ under which $\om=\Om:\pi_2(E)\cong\pi_2(M)\to\Z$. We conclude that 
\begin{equation}
E(v)=\int_{S^2}v^*\om=\om(v)=\Om(\gamma_-\#u\#\bar\gamma_+)=0\;.
\end{equation}
Therefore, $v$ is constant and $u$ is contained in the fiber over this constant. This completes the proof of Step 1.\\
%
%

\noindent {\bf Step 2:} Since every element in $\M$ is contained entirely in a fiber of $\wp:E\to M$ we can use the Reeb direction and the radial direction as a coordinate system. Thus, we can write an element $u\in\M$ as
\begin{equation}
u(s,t)=\big(\gamma(b(s,t)),F(s,t)\big),
\end{equation}
where $\gamma$ is the Reeb orbit in that fiber and $F$ denotes the radial component. Implicitly, we assume that $u$ does not hit the zero section $M$ of $E$. This can be seen as follows.

By our choice of almost complex structure $J$ the zero section $M$ is a holomorphic submanifold of $E$ of codimension 2. By choice of the Hamiltonian $h_s$ the Hamiltonian vector field $X_s$ vanishes close to $M$, see Figure \ref{pic:h} and \eqref{eqn:def_of_h_s}, i.e.~$u$ is actually holomorphic near $M$. By positivity of intersection, the intersection number $u\bullet M$ is non-negative and vanishes if and only if $u(\R\times S^1)\cap M=\emptyset$. This intersection number is well-defined since asymptotically $u$ is disjoint from $M$. We claim, of course, that the intersection number vanishes. To compute this we consider the intersection number $\big(\gamma_-\#u\#\bar\gamma_+\big)\bullet M$ between a 2-sphere $\gamma_-\#u\#\bar\gamma_+$ and the closed manifold $M$. This is a usual topological intersection number and can be computed as follows
\begin{equation}
\big(\gamma_-\#u\#\bar\gamma_+\big)\bullet M=c_1^E\big(\underbrace{\wp(\gamma_-\#u\#\bar\gamma_+)}_{=v}\big)=-\om(v)=0.
\end{equation}
The first equality follows from the definitions of the first Chern class and the bundle $E$ and the second again by construction of $E$. Now, the intersection numbers $\big(\gamma_-\#u\#\bar\gamma_+\big)\bullet M$ and $u\bullet M$ agree since $\gamma_-$ and $\bar\gamma_+$, being fiber disks, each intersect $M$ transversely in one point but with opposite sign. All in all we conclude that $u\bullet M=0$ and by positivity of intersection $u(\R\times S^1)\cap M=\emptyset$ as claimed at the beginning of this step.

The Floer equation \eqref{eqn:s_dep_Floer_eqn} in these new coordinates becomes a system of PDEs for $F$ and $b$
\bea
\label{eqn: Floer coordinates}
\partial_sb+\tfrac{1}{F}\partial_tF&=0\\
\partial_sF-F\partial_tb+Fh'_s(F)&=0.
\eea
Dividing the second equation by $F$ and setting $G=\log F$, this turns into
\bea
\label{eqn:Floer exp coordinates}
\partial_sb+\partial_tG&=0\\
\partial_sG-\partial_tb+h'_s(e^G)&=0.
\eea
Due to the Morse-Bott character in Reeb direction we work in Banach spaces with exponential weights, see \cite[Appendix A]{Frauenfelder_Arnold_Givental_Conjecture} or \cite{Bourgeois_A_Morse_Bott_approach_to_contact_homology} for a full account. Here we need only a small portion which we will explain now.

In Reeb direction we need to require exponential convergence of $u$ to the asympotic periodic orbits $\gamma_\pm$ in order for the Floer equation \eqref{eqn:Floer exp coordinates} to represent a Fredholm operator. This, in turn, means that we need to require exponential convergence for $b(s,t)$ to the function $(s,t)\mapsto t$ and for $\partial_sb(s,t)$ to zero. For that we fix $\kappa_0>0$ smaller than the spectral gap of the Hessian of $\A_L$ at $\gamma_+$ resp.~of $\A_H$ at $\gamma_-$ where $\gamma_\pm$ are the asymptotic periodic orbits of the solution $u\in\M$ we are considering. Then we choose a smooth function $\kappa(s):\R\to\R$ which agrees with the function $s\mapsto \text{sign}(s)\kappa_0s$ for $|s|\geq1$. Moreover, we require $|\kappa'(s)|\leq\kappa_0$ for all $s\in \R$. In the Floer equation \eqref{eqn: Floer coordinates} and \eqref{eqn:Floer exp coordinates} we then consider only functions $b(s,t)$ such that $b(s,t)-t\in W^{1,p,\kappa_0}(\R\times S^1,\R)$, i.e.~functions $b:\R\times S^1\to \R$ such that $(b(s,t)-t)e^{\kappa(s)}\in W^{1,p}(\R\times S^1,\R)$, where $p>2$. That is, $b$ has exponential converence at $\pm\infty$ in $s$ of rate at least $\kappa_0$ to the asymptotic periodic orbit. We point out that the Banach spaces $W^{1,p,\kappa_0}(\R\times S^1,\R)$ and $W^{1,p}(\R\times S^1,\R)$ are isomorphic via the map $\Xi:\mathfrak{b}\mapsto \mathfrak{b}(s,t)e^{\kappa(s)}$. Moreover, all solutions of \eqref{eqn:Floer exp coordinates} automatically lie in $W^{1,p,\kappa_0}(\R\times S^1,\R)$ due to the usual exponential decay estimates for the Floer equation and the choice of $\kappa_0$. We point out that we may choose $\kappa_0>0$ as small as we wish. We will use this below. 

The following argument is based on an argument by Salamon-Zehnder from \cite{Salamon_Zehnder_Morse_theory_for_periodic_solutions_of_Hamiltonian_systems_and_the_Maslov_index}. We are grateful to W.~Merry for pointing us to the article \cite{Bourgeois_Oancea_An_exact_sequence_for_contact_and_symplectic_homology} by Bourgeois-Oancea who use \cite{Salamon_Zehnder_Morse_theory_for_periodic_solutions_of_Hamiltonian_systems_and_the_Maslov_index} in a similar fashion.

We linearize equation \eqref{eqn:Floer exp coordinates} in $t$-direction and set $\zeta=(\zeta_1,\zeta_2):=(\partial_tb-1,\partial_tG)$:
\bea
\partial_s\zeta_1+\partial_t\zeta_2&=0\\
\partial_s\zeta_2-\partial_t\zeta_1+e^Gh''_s(e^G)\zeta_2&=0.
\eea
Of course, $\zeta_1$ has exponential convergence to zero of rate at least $\kappa_0$. We will analyze solutions of this linearized equation, i.e.~elements in the kernel of the linearized operator. The following argument can be found in all detail in \cite[Appendix A]{Frauenfelder_Arnold_Givental_Conjecture} or \cite{Bourgeois_A_Morse_Bott_approach_to_contact_homology}. Conjugating the linearized operator with the isomorphism $\Xi:\zeta_1\mapsto \zeta_1(s,t)e^{\kappa(s)}$ turns $\partial_s\zeta_1$ into $\partial_s\zeta_1-\kappa'(s)\zeta_1$ and leaves $\partial_t\zeta$ unchanged, indeed:
\begin{equation}
e^{\kappa(s)}\cdot\partial_s\big(\zeta_1(s,t)\cdot e^{-\kappa(s)}\big)=\partial_s \zeta_1(s,t)-\kappa'(s)\zeta_1(s,t).
\end{equation}
Thus, the kernel of the linearized operator on the space $W^{1,p,\kappa_0}$ corresponds under the isomorphism $\Xi$ to all solutions of
\bea
\label{eqn:t-lin Floer}
\partial_s\zeta_1+\partial_t\zeta_2-\kappa'(s)\zeta_1&=0\\
\partial_s\zeta_2-\partial_t\zeta_1+e^Gh''_s(e^G)\zeta_2&=0
\eea
where now $\zeta_1$ and $\zeta_2$ are in $W^{1,p}$, i.e.~are \textit{not} required to have exponential decay anymore.\footnote{
The asymmetry of \eqref{eqn:t-lin Floer} in $\zeta_1$ and $\zeta_2$ is due to that fact that the radial direction is the normal direction to the Morse-Bott manifold $\Sigma$ and therefore the operator is ''Fredholm in normal direction''. In particular, we do not need to require exponential decay in normal direction. We could, though, which would lead to an additional term $-\kappa'(s)\zeta_2$ in \eqref{eqn:t-lin Floer}. The remaining argument is essentially unaffected since the matrix norm is changed from $\|\kappa'\|+\|Fh''_s(F)\|$ to $\|\kappa'\|+\|Fh''_s(F)-\kappa'\|$ which we still can arrange to be strictly less than 1. 
} 
Combining the two equations above into a vector equation for $\zeta=(\zeta_1,\zeta_2)$ and switching back from $G$ to $F$ gives
\begin{equation}\label{eqn:lin_Floer_matrix}
\partial_s\zeta+\begin{pmatrix}0&-1\\1&0\end{pmatrix}\partial_t\zeta+
\begin{pmatrix}
-\kappa' & 0\\
0 & Fh_s''(F)
\end{pmatrix}
\zeta=0.
\end{equation}
As pointed out above we may choose the constant $\kappa_0>0$ as small as we like. The matrix norm of $\begin{pmatrix}
-\kappa' & 0\\
0 & Fh_s''(F)
\end{pmatrix}$ equals $\|\kappa'\|+\|Fh''_s(F)\|$. Using our choice of the Hamiltonian function, namely $|rh_s''(r)|<1$ for all $r$, and choosing $\|\kappa'\|\leq\kappa_0$ sufficiently small we can arrange that this matrix norm is strictly less than 1:
\begin{equation}
\|\kappa'\|+\|Fh''_s(F)\|<1\;.
\end{equation}
Now we are in the position to apply \cite[Proposition 4.2]{Salamon_Zehnder_Morse_theory_for_periodic_solutions_of_Hamiltonian_systems_and_the_Maslov_index}, which asserts that $\zeta$ must be independent of $t$.
Thus \eqref{eqn:t-lin Floer} simplifies to
\bea\label{eqn:t-lin_after_SZ}
\partial_t\zeta_1&=0\\
\partial_t\zeta_2&=0\\
\partial_s\zeta_1-\kappa'(s)\zeta_1&=0\\
\partial_s\zeta_2+Fh''_s(F)\zeta_2&=0.
\eea
The first and third equation imply that $\zeta_1\equiv0$ since for $s\geq1$ it is of the form $s\mapsto a_1e^{\kappa_0s}$ and for $s\leq -1$ of the form $s\mapsto a_2e^{-\kappa_0s}$, $a_1,a_2\in\R$, neither of which is an $L^2$-function unless $\zeta_1\equiv0$.

The second equation says that  $\zeta_2$ is independent of $t$. Since $H$ is a Morse-Bott Hamiltonian, the asymptotic periodic orbit $\lim_{s\to-\infty}u=:\gamma_-$ sits in a critical manifold diffeomorphic to $\Sigma$, see above. In particular, the Morse-Bott property implies exponential convergence of $u=(\gamma,F)$ and all its derivatives to $\gamma_-$ in normal direction. The normal direction coincides here with the radial direction. In other words $F(s)$ converges exponentially fast to $r=A$ and all its derivatives converge exponentially fast to $0$. Therefore $\zeta_2=\partial_tG=\tfrac1F \partial_tF$ also converges to $0$, that is 
\begin{equation}
\lim_{s\to-\infty}\zeta_2=0.
\end{equation}
Since $\zeta_2$ is independent of $t$ the last equation in \eqref{eqn:t-lin_after_SZ} is now an ODE for $\zeta_2(s)$. For $s\to-\infty$ the coefficient in the 0-th order term $Fh''_s(F)$ becomes $s$-independent and converges to $Ah''(A)=c\in(0,1)$, see \eqref{bound Hessian}. Therefore, asymptotically, we have
\begin{equation}
\zeta_2\sim e^{-cs}\quad\text{as}\quad s\to-\infty.
\end{equation}
Together with the vanishing asymptotic condition for $\zeta$, this implies that $\zeta_2\equiv 0$.

Going back to the original equation in $b$ and $G$, we now have found that $\partial_tG=\zeta_2=0$. This shows that $G$, and therefore $F=e^G$, is independent of $t$. 

For $b$, we now use the first equation in \eqref{eqn:Floer exp coordinates} to find that $\partial_sb=0$. By the above argument, we know that $\partial_tb-1=\zeta_1=const$ and therefore, we have
\beq
b(s,t)=const\cdot t+b(0).
\eeq
As $u$ converges to the Reeb orbit $\gamma(t)$, this asymptotic condition implies that 
\begin{equation}
b(s,t)=t\quad\forall t.
\end{equation}
This completes the proof of Step 2. The details of $F$ will be studied in Step 3.\\

\noindent {\bf Step 3:} In this step, we prove existence and uniqueness of Floer trajectories in the moduli space $\M$ in the fiber over a given Reeb orbit $\gamma$. Step 2 reduces the Floer equation \eqref{eqn:s_dep_Floer_eqn}, see also \eqref{eqn: Floer coordinates},  to a 1-dimensional ODE for $F$
\bea
\label{eqn: Floer ODE}
\partial_sF&=-F(s)(h'_s(F)-1)\\
\lim_{s\to -\infty}F(s)&=A\\
\lim_{s\to\infty}F(s)&=BR_0.
\eea
We want to show existence and uniqueness for $F$. For this we use a phase space analysis at the boundaries. 

We note that for $s<-1$, the function $h_s(r)=h(r)$ is independent of $s$ and that $h'(r)$ is non-decreasing on the interval $r\leq B+\epsilon$ with $h'(A)=1$. Therefore, for $s<-1$, the function $F(s)\equiv A$ is a solution.

Now we show that no other function solves the ODE problem for $s<-1$. By the asymptotic condition at $s=-\infty$, the function $F(s)$ is less than $B$ for some $s_0<-1$ since $A<B$. 

If $F(s_0)<A$, then $h'(s_0)<1$ and therefore, the coefficient of $F$ in \eqref{eqn: Floer ODE} is positive. This shows that $F$ is increasing in $s$. In turn, as $s$ decreases, $F(s)$ decreases further and further and this contradicts the asymptotic condition $\lim_{s\to -\infty}F(s)=A$.

Similarly, if $F(s_0)>A$, then $h'(s_0)>1$ and therefore, the coefficient of $F$ in \eqref{eqn: Floer ODE} is negative. This shows that $F$ is decreasing in $s$. Again in turn, as $s$ decreases, $F(s)$ is increasing and this contradicts again the asymptotic condition $\lim_{s\to -\infty}F(s)=A$.

Combined, this shows that for $s<-1$, the only solution satisfying the asymptotic condition at $-\infty$ is the constant solution $F(s)\equiv A$. Therefore, we can turn \eqref{eqn: Floer ODE} into an initial value problem. In particular, there is a unique maximal solution to the ODE with asymptotic condition $\lim_{s\to -\infty}F(s)=A$. It remains to check that this maximal solution is defined on $\R$ and satisfies the asymptotic condition as $s\to\infty$. 

For this we first switch again to $G=\log F$. The ODE for $G$ is then
\begin{equation}\label{eqn:ODE_for_G}
G'(s)=1-h_s'(e^G)
\end{equation}
and we have $G(s)=\log A$ for $s\leq -1$. By construction of the Hamiltonian function, $0\leq h'_s(r)<2$ for all $s$ and $r$. In particular, the maximal solution is defined on $\R$. 

From here on, we can again use a phase space analysis for the behavior for $s\geq0$, where $h_s(r)=l(r)$.
Our choice of $c$ and $B$ imply that we have 
\begin{equation}
G(0)\leq\log A+1<\log R_0B.
\end{equation} 

Therefore, we have $F(0)<R_0B$ and therefore, we have $h_s'\big(F(0)\big)<1$. This implies that $F$ is increasing and therefore converging to the next value where $h'(r)=1$ which is $r=R_0B$.
%
This proves the desired asymptotic behavior of our solution and therefore existence and uniqueness of a Floer trajectory in every fiber and therefore also that $\M\cong\Sigma$ and compactness of $\M$.\\

\noindent {\bf Step 4:}  It remains to prove Fredholm regularity of the Floer trajectories $u(s,t)=\big(\gamma(t),F(s)\big)$ studied above; i.e.~we need to show that the Fredholm operator given by the linearized Floer equation is surjective. 

%
%

We recall the Floer equation \eqref{eqn: Floer coordinates} for two functions $b(s,t), F(s,t):S^1\times \R\to\R$ is
\bea
\partial_sb+\tfrac{1}{F}\partial_tF&=0\\
\partial_sF-F\partial_tb+Fh'_s(F)&=0.
\eea
We already proved in Step 2 that $b(s,t)=t$ and that $F(s,t)$ is independent of $t$.Therefore the Floer equation reduces to
\begin{equation}
\begin{aligned}
\partial_t F&=0\\
\partial_sF+\big(h'_s(F)-1\big)F&=0.
\end{aligned}
\end{equation}
Therefore, the linearized operator is 
\begin{equation}
X(s,t)\mapsto\Big(\partial_t X,\partial_sX+\big(h''_s(F)F+h'_s(F)-1\big)X\Big)\;.
\end{equation}
Proving that this operator is surjective is equivalent to proving that the formal adjoint is injective. That is, we need to prove that the only solution to the equations
\begin{equation}
\begin{aligned}
\partial_tX&=0\\
\partial_sX-\big(h''_s(F)F+h'_s(F)-1\big)X&=0
\end{aligned}
\end{equation}
is $X=0$. Of course, the first equation implies that $X$ is independent of $t$ and we again have an ODE. Now from the definition of $h_s$ and the properties of $F$, see equations \eqref{eqn:def_of_h_s} and \eqref{eqn: Floer ODE}, we conclude for $s$ very large and positive that 
\begin{equation}
h''_s(F)F+h'_s(F)-1=\tfrac{1}{R_0^2}h''\big(\tfrac{F}{R_0}\big)F+\tfrac{1}{R_0}h'\big(\tfrac{F}{R_0}\big)-1
\end{equation}
converges to 
\begin{equation}
\tfrac{B}{R_0}h''(B)>0.
\end{equation}
In particular, $X$ solves for large $s$ an equation of the form
\begin{equation}
\partial_sX-\tilde\kappa(s)X=0
\end{equation}
with $\tilde\kappa(s)>0$ and $\lim_{s\to\infty}\tilde\kappa(s)>0$. Thus, for $s\to\infty$, $X$ is exponentially growing unless it is constant. Since $X$ is an $L^2$-function it necessarily vanishes. Finally, $X(s)$ solves an ODE and therefore has to vanish identically as we were required to prove.

This establishes Fredholm regularity of the unique Floer trajectory in each fiber and completes the proof of Theorem \ref{thm:moduli space}.
\end{proof}

\subsection{The pinched contact form}

We now consider a contact form on $\Sigma$ induced by the embedding of $\Sigma$ as graph of a function $f\colon \Sigma\to\R_{>0}$ in the complex line bundle $E\to M$, i.e., the contact form $\alpha$ on the hypersurface $\Sigma_f$.
That the contact form is pinched between two multiples of the standard contact form above is reflected by the condition
\begin{equation}
1\leq f(x)\leq R_0
\end{equation}
for all $x\in\Sigma$. Studying the Reeb flow of $\alpha$ on $\Sigma_f$ is equivalent to studying the Hamiltonian dynamics of $h_f(x,r):=h\big(\frac{r}{f(x)}\big)$ on $E$, for which $\Sigma_f$ is a level set.

We now show for completeness that the 1-periodic orbits of this Hamiltonian also correspond to periodic Reeb orbits of the contact form $\alpha_f=f\alpha$ on $\Sigma$.

As a first step, we define a vector field $V_f$ on $\Sigma$ by
\begin{equation}
\label{eq:cont_ham}
	\alpha(V_f) = 0, \qquad df(R)\alpha - df = d \alpha(V_f, \cdot),
\end{equation}
where $\alpha$ is the standard contact form on $\Sigma$, i.e., $V_f$ is contained in the contact distribution, where the second equation uniquely defines the vector field.

\begin{Lemma}
\label{lem:Ham=Reeb}
The Hamiltonian vector field $X_{h_f}$ of $h_f$ is given by 
\begin{equation}
\label{eq:symp_grad}
	X_{h_f}(x,r) = \frac{h'\big( \frac{r}{f(x)} \big) }{f(x)^2} \Big( f(x) R(x) -  V_f(x) + r df(x)[R(x)] \partial_r \Big) .
\end{equation}
If $\gamma(t) = \big(x(t),r(t)\big)$ is a 1-periodic orbit of $X_{h_f}$ then $r(t) =\bar{c} f\big(x(t)\big)$ for some constant $\bar{c}$ and we define the curves $z(t) := x\big(t/ h'(\bar{c})\big)$. 
With these definitions, $z$ is a periodic orbit of $R_f$ of period $h'(\bar{c})$, where $R_f$ is the Reeb vector field on $\Sigma$ defined by $\alpha_f$.
\end{Lemma}

In particular, the relation between periodic Reeb orbits on $\Sigma_f$, which are periodic orbits of the Hamiltonian $h_f$, and periodic Reeb orbits is given by radial projection.

\begin{proof}
The formula for the Hamiltonian vector field is checked by computing $dh_f$ and plugging $X_{h_f}$ into $\omega=d(r\alpha)$. This definition uses the natural splitting of the tangent space into the radial component, the Reeb direction and the contact distribution.
For the 1-form $dh_f$, we have 
\[
	dh_f(x,r)[ v + a \partial_r] = h'\left( \frac{r}{f(x)} \right) \left(\frac{a}{f(x)}- \frac{r}{f(x)^2}df(x)[v]\right),
\]
where $v$ is a tangent vector to $\Sigma$ and $a\in\R$. We now compute $i_{X_{h_f}}d(r\alpha)(v+a\partial_r)$ using the expression for the Hamiltonian vector field as stated.
\begin{align*}
d (r\alpha) \big(X_{h_f}(x,r), (v+a\partial_r\big) & = (d r \wedge \alpha + r d \alpha) \big(X_{h_f}(x,r), v + a \partial_r \big) \\
	& = \frac{h'\big( \frac{r}{f(x)} \big) }{f(x)^2}  (d r \wedge \alpha + r d \alpha) \Big( f(x)R(x) - V_f(x) + r df(x)[R(x)] \partial_r , v + a \partial_r)\Big) \\
	& =  \frac{h'\big( \frac{r}{f(x)} \big) }{f(x)^2} \Big(  r df(x)[R(x)] \alpha(v) - f(x) a - r d \alpha(V_f(x), v)\Big) \\
	& \stackrel{(*)}{=} \frac{h'\big( \frac{r}{f(x)} \big) }{f(x)^2} \Big(   r df(x)[R(x)] \alpha(v) - f(x) a - r  df(x)[R(x)]\alpha(v) + r df(x)[v] \Big) \\
	& = \frac{h'\big( \frac{r}{f(x)} \big) }{f(x)^2} ( - f(x)a +  r df(x)[v]) \\
	& = - dh_f(x,r)[ v + a \partial_r],
\end{align*}
where $(*)$ uses the second equation in  \eqref{eq:cont_ham}. This shows that \eqref{eq:symp_grad} indeed is the Hamiltonian vector field.

Since $h_f$ is autonomous, the fact that $\gamma$ is a 1-periodic orbit of $X_{h_f}$ implies that $h_f\big(\gamma(t)\big)$ is constant. Thus if $\gamma(t) = \big(x(t),r(t)\big)$ then $r(t)/ f\big(x(t)\big)$ is constant, since $h$ is strictly increasing. Thus $\gamma(t) =\big(x(t), \bar{c} f(x(t)\big)$ for some contant $\bar{c}$. 

Set $z(t) := x\big(t/h'(\bar{c})\big)$ and 
we claim that $z$ is a periodic Reeb orbit of $R_f$. For this, we compute 
\begin{equation}
 \dot z (t)= \frac{1}{h'(\bar{c})} \dot x \big(t/h'(\bar{c})\big) = \frac{1}{f\big(z(t)\big)}R\big(z(t)\big) - \frac{1}{f\big(z(t)\big)^2}V_f\big(z(t)\big)
 \end{equation}
from \eqref{eq:symp_grad}. 
Thus to complete the proof it suffices to show that
\begin{equation}
\label{eq:R_sigma}
	R_f(x) = \frac{1}{f(x)}R(x) - \frac{1}{f(x)^2}V_f(x) .
\end{equation}
is the Reeb vector field of $\alpha_f$ on $\Sigma$.
 Clearly, one has $\alpha_f(R_f) = 1$. Now 
\begin{align*}
 d \alpha_f(R_f, \cdot) & = (d f \wedge \alpha + f d \alpha)\left( \frac{1}{f(x)}R(x) - \frac{1}{f(x)^2}V_f(x) , \cdot \right) \\
 & = \frac{1}{f}df(R) \alpha  - \frac{1}{f} df  - \frac{1}{f^2}df(V_f) \alpha  + \frac{1}{f^2} \alpha(V_f) df - \frac{1}{f}d \alpha(V_f,\cdot) \\
 & =  - \frac{1}{f^2} df(V_f) \alpha= 0 ,
\end{align*}
where we used both equations in  \eqref{eq:cont_ham} again. Furthermore, we used that $df(V_f) = 0$ which can be seen by feeding $V_f$ to both sides of the second equation of \eqref{eq:cont_ham}.
\end{proof}

The function $h$ above is monotone increasing and therefore, this pinching condition also implies that 
\beq
\label{eqn:pinching homotopy}
h(r)\geq h\left(\frac{r}{f(x)}\right)\geq h\left(\frac{r}{R_0}\right).
\eeq
We now want to define a homotopy from $h(r)$ to $h\big(\frac{r}{R_0}\big)$ that is not strictly radial as above, but passes through $h\big(\frac{r}{f(x)}\big)$ instead.


Similar to the function $\beta$ above, we now define three functions $\beta_1,\beta_2^\rho, \beta_3^\rho\colon \R\to[0,1]$ depending smoothly on a parameter $\rho>0$ such that 
\bea
\beta_1(s)+\beta_2^\rho(s)+ \beta_3^\rho(s)=1 &\quad \forall s\in\R\\
\beta_1(s)\equiv 1&\quad  \forall s\leq-1\\
\beta_2^\rho(s)\equiv 1&\quad \forall s\in (0,\rho k)\\
\beta_3^\rho(s)\equiv 1&\quad  \forall s\geq\rho k+1.
\eea
Furthermore, we require $\beta_1$ to be monotone decreasing and $\beta_3^\rho$ to be monotone increasing. For $\rho=0$, we choose $\beta_2^0\equiv 0$ and $\beta_1=\beta$, where $\beta$ is the function used above for the radial homotopy and $\beta_i$ depend smoothly on $\rho$. Furthermore, we require the convergence as $\rho\to 0$ to be a $C^\infty_{loc}$-convergence of $\beta_i^\rho$ to the specified functions $\beta_i^0$.

Now consider the homotopy
\beq
H_s^\rho(x,r)=\beta_1(s)h(r)+\beta_2^\rho(s)h\left(\frac{r}{f(x)}\right)+\beta_3^\rho(s)h\left(\frac{r}{R_0}\right).
\eeq
The pinching condition \eqref{eqn:pinching homotopy} implies that with this choice, we have
\begin{equation}
\frac{\partial H_s^\rho}{\partial s}\leq 0.
\end{equation}
Therefore, the action estimate
\beq
E(u)\leq \A\bigl(u(+\infty)\bigr)-\A\bigl(u(-\infty)\bigr)
\eeq
holds for all solutions to the Floer equation
\beq
\label{eqn:Floer_eqn_H_s} 
\partial_su+J(u)\big(\partial_tu-X^\rho_s(u)\big)=0
\eeq
with finite energy, where $X_s^\rho$ is now the Hamiltonian vector field of $H_s^\rho$.

\section{Proof of Theorem \ref{Main Theorem}}

Now we are in a position to prove our main theorem. The idea is inspired by previous work of two of the authors in \cite{AH}, which in turn is based on the first authors work in \cite{Albers_Momin_Cup_length_estimates_for_leafwise_intersections}.
If the Hamiltonian vector field of $h\big(\frac{r}{f(x)}\big)$ has infinitely many periodic orbits then the Reeb vector field of $\alpha_f$ necessarily has infinitely many simply periodic orbits, simply because $h'$ is bounded. Therefore, from now on we assume that the Hamiltonian vector field of $h\big(\frac{r}{f(x)}\big)$ has finitely many periodic orbits.

The main difference with \cite{AH} is that we do not assume that the symplectic manifold $(E,\Om)$ has any kind of nice behavior concerning bubbling off of holomorphic sphere. In fact, even though $(M,\om)$ satisfies $[\om]\in\H^2(M;\Z)$ the manifold $(E,\Om)$ is not necessarily semi-positive. We will rule out bubbling-off of holomorphic spheres by a simple energy argument instead. The non-compactness of $E$ poses no problem, since it is convex at infinity.
Furthermore, we need some additional marking structures to make sure that our moduli spaces have the correct dimension.

\subsection{Defining the moduli spaces}

We first explain the philosophy behind these technical constructions below. The main idea is to vary the moduli space studied in Section \ref{sec:initial_moduli_space} with a parameter $\rho$ and find the desired periodic orbits by enforcing breaking in certain limits. There are two issues that require technical solutions: the first one is that the Reeb flow is autonomous and we need to break the $S^1$-symmetry. To do this we fix finitely many parametrizations of the periodic orbits of $\alpha_f$. This is done by marking the angle coordinate of $\gamma(0)$ in polar coordinates in the fiber over $\wp\big(\gamma(0)\big)$ where $\gamma$ is such a periodic orbit. 


The second issue is that we want to connect this moduli space of Floer cylinders with a Morse-theoretic picture of the cup-product and then use the parameter $\rho$ to define chain homotopy equivalent homology operations. For this we want to study cylinders over intersections of stable and unstable manifolds of critical points for certain Morse functions. The marking condition on the Reeb orbits above needs to be translated into a marking condition for certain points on the Floer cylinders which move with varying $\rho$, but stay above the stable and unstable manifolds. As we want to consistently keep the markings we consider trivializations of the bundle $E$ over the stable and unstable manifolds. Of course, these trivializations need to be compatible with trivializations over the periodic Reeb orbits.

Now we describe the precise setup. The first step is to trivialize the bundle over each of the projections of the finitely many Reeb orbits. Now, we choose a generic $\theta_0\in S^1$ such that all periodic Reeb orbits meet the ray $\R_+\cdot \theta_0$ (in the chosen trivialization) in only finitely many points. This is possible for a generic choice of $\theta_0$. We are interested in the projections of such points to $M$ and denote the collection of these points from all Reeb orbits by $p_1, \ldots, p_\nu\in M$.

Choose generic Morse functions $f_\ast, f_1,\ldots, f_k$ on $M$ such that there are critical points $x_\ast^\pm\in \Crit f_\ast$ of $f_\ast$ and $x_i\in \Crit f_i$ of $f_i$ for $i=1,\ldots, k$ corresponding to cohomology classes whose cup-product is non-zero. In particular, $k\leq\operatorname{cuplength}(M;\Z/2)$. We assume, from now on that $k=\operatorname{cuplength}(M;\Z/2)$ even so everything works for $k\leq\operatorname{cuplength}(M;\Z/2)$.
We refer to \cite{Schwarz_Morse_homology} for details on the Morse theoretic cup-product and note here only that being non-zero implies the stable manifolds of $x_1, \ldots, x_k$ have non-empty intersection, i.e., there are Morse trajectories $\eta_i$ converging to $x_i$ such that all $\eta_i(0)$ agree and $\eta_i(0)\in W^u(x_\ast^-,f_\ast)\cap W^s(x_\ast^+,f_\ast)$. Denote the gradient flow lines from $\eta_i(0)$ to $x_\ast^\pm$ in positive and negative direction by $\eta_\pm$. We call this a bouquet of gradient flow lines, see Figure \ref{fig:Moduli_rho=0}.

%

We now start building the moduli space we want to study and add some generic conditions for the functions $f_i,f_\ast$. The first step is to consider
\beq
\widehat\Ma:=\left\{(\rho,u)\,\left|\; 
\begin{aligned}
\rho\geq 0,\, u=(\gamma,F) \text{ solves \eqref{eqn:Floer_eqn_H_s}} \\
F(-\infty)=A,  F(+\infty)=R_0B\\
\Om\big(u(-\infty)\#u\#\bar u(+\infty)\big)=0
\end{aligned}
\right.\right\} .
\eeq
Here again we use the convention that the periodic orbits $u(\pm\infty)$ are capped by their fiber disk. The bar in $\bar u(+\infty)$ indicates that the orientation of the fiber disk is reversed. At the boundary of $\widehat\Ma$, i.e.~for $\rho=0$, we have
\beq
\partial\widehat\Ma:=\widehat\Ma|_{\rho=0}=\M\cong\Sigma,
\eeq
where $\M$ is the moduli space studied above, see \eqref{eqn:moduli_space_continuation_traj} and Theorem \ref{thm:moduli space}. Indeed, for $\rho=0$ the Hamiltonian $H_s^0$ agrees with $h_s$ from above, where we have established Fredholm regularity for this moduli space.

We now add the bouquet of gradient flow lines to the picture. Roughly speaking the idea is to single out elements $u$ in $\widehat\Ma$ which lie over the bouquet in a prescribed manner. The Morse bouquet is an intersection of stable and unstable manifolds 
\begin{equation}
\label{eqn: Bouquet intersection}
W^s(x_*^+,f_*)\cap W^u(x_\ast^-, f_\ast)\cap W^s(x_1,f_1)\cap\cdots\cap W^s(x_k,f_k)
\end{equation}
and consists of a finite number of points (in fact, an odd number). Since stable and unstable manifolds are contractible the $\C$-bundle $E$ resp.~$S^1$-bundle $\Sigma$ is trivial over each of these manifolds. We fix trivializations over each of the above stable/unstable manifolds such that over the finitely many points in $W^s(x_*^+,f_*)\cap W^u(x_\ast^-, f_\ast)\cap W^s(x_1,f_1)\cap\cdots\cap W^s(x_k,f_k)$ all trivializations agree. This is possible since there is no obstruction to extending a trivialization of an $S^1$-bundle over finitely many points to a 1-dimensional CW complex. Of course, in general the trivializations over the various stable/unstable manifolds only match up precisely at the finitely many intersection points, the Morse bouquets.
Note that generically, the stable manifolds $W^s(x_i,f_i)$ do not meet the points $p_1,\ldots, p_\nu$ as the index and therefore the codimension of the stable manifold of $x_i$ is at least one. Furthermore, the stable manifolds meet the projections of the Reeb orbits in finitely many points. We start building the trivializations starting from these finitely many intersection points such that the trivializations at these points agree with the trivialization over the Reeb orbits chosen above.
With these choices of trivializations, we have the following properties:
\begin{itemize}
\item[(A)] \label{a} Whenever $\wp\big(\gamma(0)\big)\in W^s(x_i,f_i)$ for some Reeb orbit $\gamma$ of $\alpha_f$ then the trivializations of $E$ over $\wp(\gamma)$ and over the stable manifold $W^s(x_i,f_i)$ agree at $\wp\big(\gamma(0)\big)$. Moreover, in this trivialization we have $\arg \gamma(0)\neq\theta_0$.
\item[(B)] The intersection \eqref{eqn: Bouquet intersection} consists of finitely many points and the trivializations over the stable and unstable manifolds agree over those points.
\end{itemize}

In fact, talking about stable/unstable manifolds we implicitly chose Riemannian metrics $g_*,g_1,\ldots,g_k$ on $M$. We assume that these metrics are so that all intersections of stable/unstable manifolds are transverse. This is a generic property for the pairs $(f_i,g_i)$ and $(f_\ast,g_\ast)$. Since we assume that the critical points come from a non-vanishing cup-product the above intersection is a manifold of dimension zero and of odd cardinality.


 Now we define the moduli space of interest for the proof as 
\beq\label{eqn:terminal_moduli_space}
\Ma:=\left\{(\rho,u)\in\widehat\Ma\;\left| 
\begin{aligned}
\;\;&\wp\big(u(i\rho,0)\big)\in W^s( x_i, f_i)\;\text{ for } i=1,\ldots, k,\\
&\wp\big(u(0,0)\big)\in W^u( x_\ast^-, f_\ast),\; \wp\Big(u\big((k+1)\rho,0\big)\Big)\in W^s( x_\ast^+, f_\ast),\\ 
&\arg u(i\rho,0)=\theta_0\;\text{ for } i=0,\ldots, k+1\\
\end{aligned}
\right.\right\}.
\eeq
The angle in the last condition is understood as the angle in the trivialization of the bundle over the stable or unstable manifold from the other conditions, i.e., over $W^s( x_i, f_i)$ for $ i=1,\ldots, k$, over $W^u( x_\ast^-, f_\ast)$ for $i=0$ and over $W^s( x_\ast^+, f_\ast)$ for $i=k+1$.

Note that the periodic Reeb orbits we are interested in and that we used above to construct the trivializations correspond to periodic orbits of the Hamiltonian $h_f$ by radial projection. In particular, the projections to $M$ and the angle coordinates in the chosen trivializations agree. Thus from now on we can work in the Hamiltonian setting and still have conditions (A) and (B) for the choices of trivializations and marking.

Let us discuss this first for $\rho=0$.
In this case, $u(0,0)=u(i\rho,0)=u\big((k+1)\rho,0\big)$. Thus,
the conditions  in \eqref{eqn:terminal_moduli_space} are picking out those solutions $u$ to the Floer equation \eqref{eqn:Floer_eqn_H_s} (which actually for $\rho=0$ agrees with \eqref{eqn:s_dep_Floer_eqn}) which are parametrized such that $\arg u(0,0)=\theta_0$. That such a configuration is Fredholm regular is proved below. For $\rho>0$ both the Hamiltonian term in the Floer equation and the incidence conditions start to change, see Figure \ref{fig:Moduli_rho_positive}. 
\begin{figure}
\centering
\begin{tikzpicture}[>=stealth]
\coordinate (A) at (0,7);
\coordinate (B) at (0,9);
\coordinate (C) at (10,7);
\coordinate (D) at (10,9);
\node[below left] (E) at (5,8) {$u(0,0)$};
\coordinate[label=above:$x_\ast^+$ ] (F01) at (1.5,4); 
\coordinate[label=below:$x_\ast^-$ ] (F02) at (5,-2); 
\draw (5,1) to node[midway](f) {} (5,-2); 
\coordinate[label=above:$x_1$ ] (F1) at (3,4); 
\coordinate[label=above:$x_2$ ] (F2) at (4,4); 
\coordinate[label=above:$x_{k-1}$ ] (F3) at (6,4); 
\coordinate[label=above:$x_k$ ] (F4) at (7,4);
\coordinate[label=above:$\cdots$ ] (F5) at (5,3.5);
\draw (A) to (C);
\draw (B) to (D);
\draw[bend left=30] (A) to (B);
\draw[bend right=30, dotted] (A) to (B);
\draw[bend left=30] (C) to (D);
\draw[bend right=30] (C) to (D);
\draw (1.5,4) to node[near start](e) {} (5,1);
\draw (F1) to node[near start](a) {} (5,1);
\draw (F2) to node[near start](b) {} (5,1);
\draw (F3) to node[near start](c) {} (5,1);
\draw (F4) to node[near start](d) {} (5,1);
\draw [->] (5,1) to (a);
\draw [->] (5,1) to (b);
\draw [->] (5,1) to (c);
\draw [->] (5,1) to (d);
\draw [->] (5,1) to (e);
\draw [->] (5,-2) to (f);
\fill (5,1) circle (.08cm);
\fill (5,8) circle (.05cm);
\node[right] (g) at (5,1) {$\wp\bigl(u(0,0)\bigr)$};
\draw[->, very thick] (5,6.5) to (5,4.5);
\node[right] at (5,5.6) {$\wp$};
\end{tikzpicture}
\caption{The moduli space at $\rho=0$} \label{fig:Moduli_rho=0}
\end{figure}
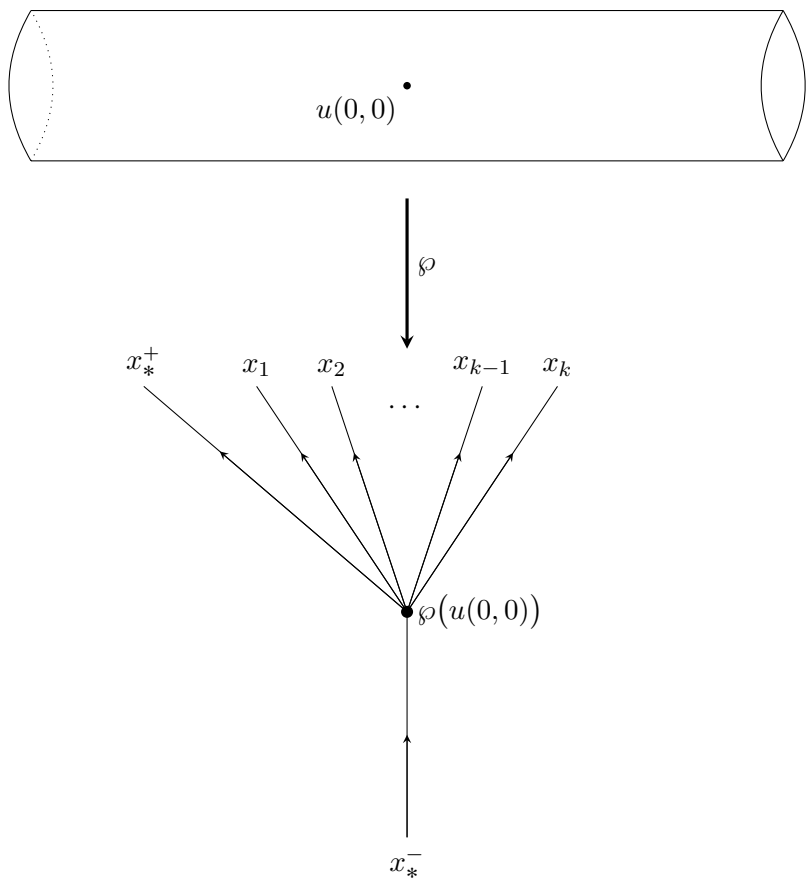

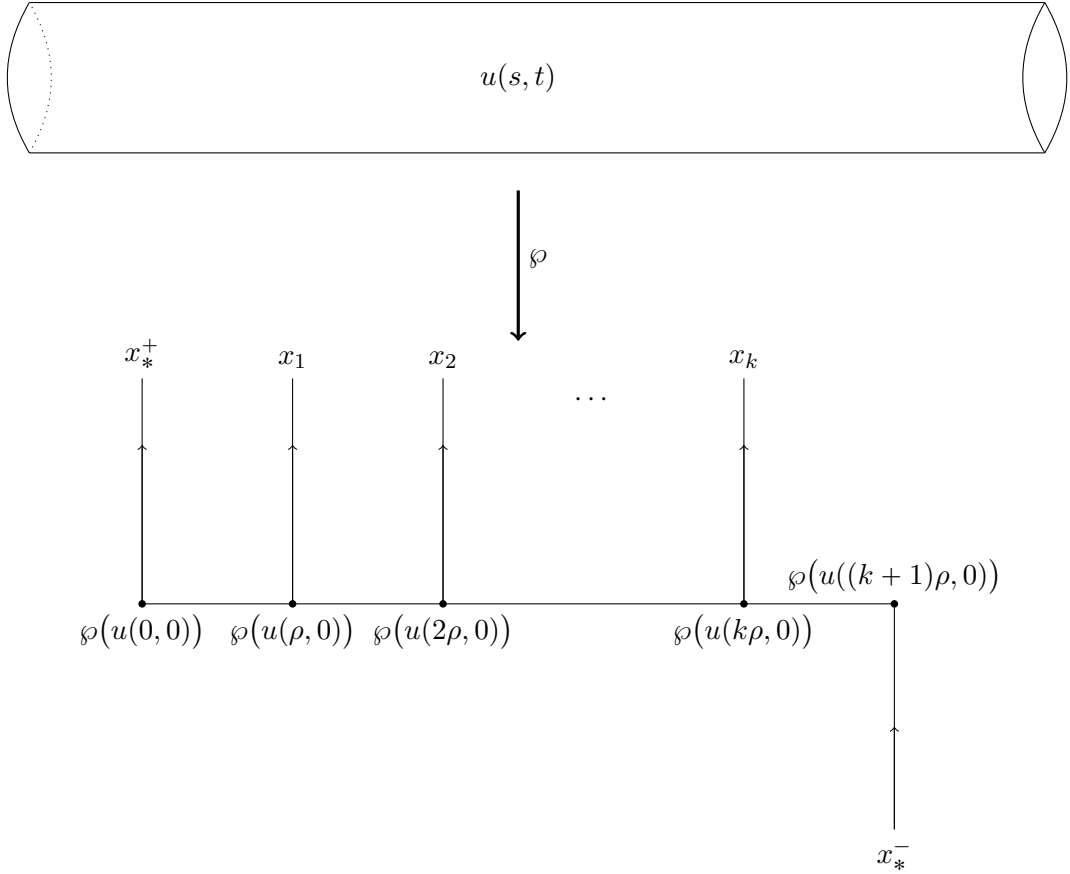
\begin{figure}
\begin{center}
  \begin{tikzpicture}
\coordinate (A) at (-1.5,7);
\coordinate (B) at (-1.5,9);
\coordinate (C) at (12,7);
\coordinate (D) at (12,9);
\coordinate (E) at (5,1);
\draw (A) to (C);
\draw (B) to (D);
\draw[bend left=30] (A) to (B);
\draw[bend right=30, dotted] (A) to (B);
\draw[bend left=30] (C) to (D);
\draw[bend right=30] (C) to (D);
\node at (5,8) {$u(s,t)$};
\draw[->, very thick] (5,6.5) to (5,4.5);
\node[right] at (5,5.6) {$\wp$};
\coordinate[label=above:$x_\ast^+$ ] (F01) at (0,4); 
\coordinate[label=above:$x_1$ ] (F1) at (2,4); 
\coordinate[label=above:$x_2$ ] (F2) at (4,4); 
\coordinate[label=above:$x_k$ ] (F4) at (8,4);
\coordinate[label=above:$\cdots$ ] (F5) at (6,3.5);
\coordinate[label=below:$x_\ast^-$ ] (F02) at (10,-2);
\draw (10,1) to node[midway](a) {} (10,-2); 
\draw (0,4) to node[near start](b) {} (0,1);
\draw (F1) to node[near start](c) {} (2,1);
\draw (F2) to node[near start](d) {} (4,1);
\draw (F4) to node[near start](e) {} (8,1);
\draw [->] (10,-2) to (a);
\draw [->] (0,1) to (b);
\draw [->] (2,1) to (c);
\draw [->] (4,1) to (d);
\draw [->] (8,1) to (e);
\draw (0,1) to (10,1);
\fill (0,1) circle (.05cm);
\fill (2,1) circle (.05cm);
\fill (4,1) circle (.05cm);
\fill (8,1) circle (.05cm);
\fill (10,1) circle (.05cm);
\node[below] (G1) at (2,1) {$\wp\bigl(u(\rho,0)\bigr)$};
\node[below] (G2) at (4,1) {$\wp\bigl(u(2\rho,0)\bigr)$};
\node[below] () at (0,1) {$\wp\bigl(u(0,0)\bigr)$};
\node[below] (G4) at (8,1) {$\wp\bigl(u(k\rho,0)\bigr)$};
\node[above] (E02) at (10,1) {$\wp\bigl(u((k+1)\rho,0)\bigr)$};
\end{tikzpicture}
\caption{The moduli space at $\rho>0$} \label{fig:Moduli_rho_positive}
\end{center}
\end{figure}

\begin{Prop}
The moduli space $\partial\Ma=\Ma|_{\rho=0}$ is Fredholm regular and consists of an odd number of points.
\end{Prop}
\begin{proof}
By Theorem \ref{thm:moduli space}, $\M=\widehat\Ma|_{\rho=0}\cong\Sigma$ consists of Fredholm regular solutions to the Floer equation \eqref{eqn:s_dep_Floer_eqn} and is equivariantly diffeomorphic to $\Sigma$. In addition, the bouquet from Morse gradient flow lines is Fredholm regular by assumption. Moreover, the conditions in the definition of $\Ma$ for $\rho=0$ simplify to
\begin{equation}
\left\{\;\;
\begin{aligned}
&\wp\big(u(0,0)\big)\in W^s(x_*^+,f_*)\cap W^u(x_\ast^-, f_\ast)\cap W^s(x_1,f_1)\cap\cdots\cap W^s(x_k,f_k)\\
&\arg u(0,0)=\theta_0
\end{aligned}
\right.
\end{equation}
In other words, these conditions single out precisely one solution over each intersection point of the Morse bouquet. So, the full Fredholm problem is Fredholm regular and of index 0 as claimed and $\partial \Ma$ consists of an odd number of points.
\end{proof}

In order to obtain the Reeb orbits claimed in the main Theorem we want to force breaking of Floer trajectories in certain limits. For this, we need to show the following

\begin{Prop}
\label{prop:compact}
There exists a sequence $(\rho_n,u_n)\in\Ma$ such that $\rho_n\to\infty$.
\end{Prop}
\begin{proof}
Assume that this is not the case and
for all sequences $\bigl\{(\rho_n,u_n)\bigr\}_{n\in\N}$ in $\Ma$, the parameter $\rho$ stays bounded, i.e., we have
	\[
		\sup_{n\in\N}\rho_n<\infty.
	\] 
We show that in this case, the moduli space is compact and how this leads to a contradiction.
	Let $\bigl\{(\rho_n,u_n)\bigr\}_{n\in\N}$ be a sequence in $\Ma$. Possibly by passing to a subsequence, we can assume $\rho_n$ converges to $\rho^*$.
	We would like to apply a result by Schwarz in \cite[Proposition 4.3.11]{Schwarz_PhD} stating that convergence of Floer trajectories in $C^\infty_{loc}$ without breaking or bubbling already implies convergence in $H^{1,p}$.
	We already have the $C^{\infty}_{\textrm{loc}}$-convergence and thus need only to show that there is no bubbling nor breaking.
	
	Since $\rho_n\to\rho^*$, breaking is only possible near the ``ends'' of the cylinder.
	There are two possibilities, breaking at $+\infty$ or at $-\infty$.
	At $-\infty$, we have to break on a critical point for $h(r)$ with action in $[A\,,\,c(B-A)+A]$ since the homotopy of Hamiltonian functions is monotone.
	This is impossible as the only such point is the asymptotic orbit $\gamma_-$. The argument excluding breaking at $+\infty$ is the same using the properties of $l(r)$ and the orbit $\gamma_+$.
	
Bubbling is prevented since the energy of all elements $(\rho,u)\in\M$ curve is less than $1$. Indeed, $E(u)\leq\A(\gamma_+)-\A(\gamma_-)<1$ by construction. Therefore there is not enough energy for bubbling-off of holomorphic spheres since on $\pi_2(E)\cong\pi_2(M)$ we have $\Om\big(\pi_2(E)\big)=\om\big(\pi_2(M)\big)\subset\Z$ due to the assumption $[\om]\in\H^2(M,\Z)$. Moreover, as $E$ is convex at infinity the sequence $(u_n)$ does not escape to infinity in $E$. Indeed, we recall that $h(r)$ is constant for $r$ large and therefore all solution $u_n$ become holomorphic near infinity.

This shows that under the assumption that $\rho$ stays bounded for all sequences in $\Ma$, the moduli space $\Ma$ is compact.

By construction, the parametrized moduli space $\Ma$ has only one boundary component $\partial\Ma=\Ma|_{\rho=0}=\M$ which, as shown above, is Fredholm regular.
By compactness, $\Ma$ is still Fredholm regular for small values of $\rho$. Using an abstract perturbation argument as in \cite{AH}, see also \cite[Theorems 5.5 and 5.13]{HWZ_Polyfold_and_Fredholm_Theory_I_Basic_theory_in_M_polyfolds}, we can define a perturbed moduli space $\widetilde{\Ma}$, which is a smooth, 1-dimensional compact manifold and agrees with $\Ma$ near $\rho=0$, where $\Ma$ is already Fredholm regular.

In particular, then $\widetilde{\Ma}$ is a compact, 1-dimensional manifold with only one boundary component. As this cannot exist, we have shown that the assumption at the beginning of this proof is wrong and there exists a sequence $\bigl\{(\rho_n,u_n)\bigr\}_{n\in\N}$ of elements in $\Ma$ with $\rho_n\to\infty$.
\end{proof}

As the last step in this section, we also define moduli spaces for bounded values of $\rho$.
Namely, we set
\[
\Ma_\rho(x_1,\ldots, x_k, x_\ast^-;x_\ast^+)=\left\{u\mid (u,\rho)\in\Ma\right\}
\]
and
\[
\Ma[0,\rho]=\left\{u\mid (u,\sigma)\in\Ma\ \forall\ \sigma\in[0,\rho]\right\}.
\]
As in \cite{AH}, also these moduli spaces can be perturbed to be smooth compact manifolds $\widetilde\Ma_\rho$ for $\rho\in\N$ by an abstract perturbation argument, cf. \cite[Theorems 5.5 and 5.13]{HWZ_Polyfold_and_Fredholm_Theory_I_Basic_theory_in_M_polyfolds}. Moreover, as described above, $\Ma_0=\partial\Ma=\M$ is already Fredholm regular and the perturbations can be done leaving $\Ma_\rho$ untouched for small $\rho$. Then we can also perturb the moduli spaces $\Ma[0,\rho]$ for $\rho\in\N$ keeping the ends fixed to get smooth manifolds $\widetilde\Ma[0,\rho]$.

\subsection{Finding critical points of the action functional}

The next step is to use the above moduli spaces to construct cohomology operations.
It is rather standard, cf. \cite{AH, Schwarz_Morse_homology}, that the projection of $\Ma_0$ to $M$ defines the cup product on $M$ by
\bea
\theta_0\colon\CM^\ast(f_1)\otimes\ldots\otimes\CM^\ast(f_k)\otimes\CM_\ast(f_\ast)&\to\CM_\ast(f_\ast)\\
x_1\otimes\ldots\otimes x_k\otimes x_\ast^-&\mapsto\sum_{x_\ast^+\in\Crit f_\ast}\#_2pr_M\widetilde\Ma_0(x_1,\ldots,x_k,x_\ast^-;x_\ast^+)\cdot x_\ast^+.
\eea
Here, we use Morse homology and cohomology with coefficients in $\Z/2\Z$. Observe that the functions $f_i$ and $f_\ast$ are defined on $M$ and for $\rho=0$, the projection is a standard Morse bouquet as the cylinder projects to a point.
Thus all homology and cohomology groups above can be identified with the singular homology and cohomology of $M$. Then the left hand side corresponds to the homology class
\begin{equation}
\big([x_1]\cup\ldots\cup[x_k]\big)\cap [x_\ast^-].
\end{equation}
Furthermore, as $k=\CL(M)$, we can choose generic Morse functions such that there are critical points $x_\ast^\pm$ and $x_1,\ldots, x_k$ such that this product is non-zero. In particular, this shows that the moduli space $pr_M\widetilde\Ma_0(x_1,\ldots,x_k,x_\ast^-,x_\ast^+)$ is nonempty and therefore, we also have
\begin{equation}
\widetilde\Ma_0(x_1,\ldots,x_k,x_\ast^-;x_\ast^+)\neq\emptyset.
\end{equation}
As the next step, we define cohomology operations depending on $\rho\in\N$ by
\bea
\theta_\rho\colon\CM^\ast(f_1)\otimes\ldots\otimes\CM^\ast(f_k)\otimes\CM_\ast(f_\ast)&\to\CM_\ast(f_\ast)\\
x_1\otimes\ldots\otimes x_k\otimes x_\ast^-&\mapsto\sum_{x_\ast^+\in\Crit(f_\ast)}\#_2\widetilde\Ma_\rho(x_1,\ldots,x_k,x_\ast^-;x_\ast^+)\cdot x_\ast^+.
\eea
As in \cite{AH, Albers_Momin_Cup_length_estimates_for_leafwise_intersections}, these operations are chain homotopy equivalent to $\theta_0$ using the moduli spaces $\widetilde\Ma[0,\rho]$ to define the chain homotopy. In particular, this shows that for all $n\in\N$, there are generic Morse functions $f_i$ and $f_\ast$, possibly depending on $n$, with critical points $x_i$ and $x_\ast^\pm$ such that 
\begin{equation}
\widetilde\Ma_n(x_1,\ldots,x_k,x_\ast^-;x_\ast^+)\neq\emptyset.
\end{equation}
This implies that also
\begin{equation}
\Ma_n(x_1,\ldots,x_k,x_\ast^-;x_\ast^+)\neq\emptyset
\end{equation}
as otherwise also a small perturbation of $\Ma_n(x_1,\ldots,x_k,x_\ast^-;x_\ast^+)$ would be empty, too, and therefore the cohomology operations would vanish.

We now run the $C^\infty_{\textrm{loc}}$ compactness $k$ times by centering ourselves at each $l\rho_n$ for $l=1,\ldots,k$ where $(\rho_n,u_n)$ is a sequence guaranteed by Proposition \ref{prop:compact}. This means that we choose $u_n\in\Ma_n(x_1,\ldots,x_k,x_\ast^-;x_\ast^+)$ and consider the sequences
\begin{equation}
u_{n,l}(s,t)=u_n(s+nl,t).
\end{equation}
As in \cite{AH, Albers_Momin_Cup_length_estimates_for_leafwise_intersections}, these sequences converge to a broken Floer trajectory for $n\to\infty$ and we find $k+1$ critical points $(\gamma_i,\overline{\gamma}_i)$ of $\A_{h_f}$ and $(\gamma_-,\bar\gamma_-)$ of $\A_H$ and $(\gamma_+,\bar\gamma_+)$ of $\A_L$ such that
\begin{equation}
\A_H(\gamma_-,\overline{\gamma}_-)\leq\A_{h_f}(\gamma_1,\overline{\gamma}_1)\leq \cdots\leq\A_{h_f}(\gamma_i,\overline{\gamma}_i)\leq\cdots\leq\A_{h_f}(\gamma_{k+1},\overline{\gamma}_{k+1})\leq\A_L(\gamma_+,\overline{\gamma}_+)\;,
\end{equation}
where $\A_{h_f}$ is the action functional for the
 Hamiltonian $h_f(r,x)=h\big(\frac{r}{f(x)}\big)$ describing the Reeb flow of the contact form $\alpha_f$ on $\Sigma$ by Lemma \ref{lem:Ham=Reeb}.

Moreover, we know that $\A_H(\gamma_-,\overline{\gamma}_-)=A$ and $\A_L(\gamma_+,\overline{\gamma}_+)
=c(B-A)+A$, see \eqref{eqn:moduli_problem}. We now claim that most inequalities are strict by generic choice of Morse functions.
Indeed, if there were an equality, the corresponding Floer trajectory would have zero energy and thus be independent of the $s$-coordinate, i.e.~equal to a Reeb orbit $\gamma$ of $\alpha_f$. Then the coincidence condition in the definition of the moduli space $\Ma$ shows that we must have $\wp\big(\gamma(0)\big)\in W^s( x_j, f_j)$ and that $\arg\gamma(0)=\theta_0$ in the trivialization over $W^s( x_j, f_j)$. This contradicts condition (A) in the construction of our trivializations and generic Morse functions. Therefore, we conclude
\[
	\A_{h_f}(\gamma_i,\overline{\gamma}_i)<\A_{h_f}(\gamma_{i+1},\overline{\gamma}_{i+1})\quad\forall i=1,\ldots,k\;.
\]
A remaining possibility is that $\gamma_i=\gamma_{i+1}$ but this does not occur ($\overline{\gamma}_{i}\neq\overline{\gamma}_{i+1}$). Indeed, this is ruled out as follows. We compute
\begin{equation}
\A_{h_f}(\gamma_{i+1},\overline{\gamma}_{i+1})-\A_{h_f}(\gamma_i,\overline{\gamma}_i)=\Om(\overline{\gamma}_{i+1}\#(-\overline{\gamma}_{i}))\in\Z\setminus\{0\}
\end{equation}
since $\Om=\om:\pi_2(E)\cong\pi_2(M)\to\Z$ since $[\om]\in\H^2(M,\Z)$. On the other hand we know that 
\begin{equation}
\begin{aligned}
\A_{h_f}(\gamma_{i+1},\overline{\gamma}_{i+1})-\A_{h_f}(\gamma_i,\overline{\gamma}_i)
&\leq\A_L(\gamma_+,\overline{\gamma}_+)-\A_H(\gamma_-,\overline{\gamma}_{-})\\
&=c(B-A)<1\;.
\end{aligned}
\end{equation}
by equations \eqref{small_value} and \eqref{eqn:moduli_problem}. This contradiction shows that $\gamma_i\neq\gamma_{i+1}$. 

%
%

\section{Proofs of Corollaries}

In this section, we finally prove the statements about Reeb dynamics implied by the Theorem. 
In detail, Corollary \ref{cor:multiplicity} is an immediate consequence of Lemma \ref{lem:Ham=Reeb}. In the the special case of starshaped hypersurfaces in $\R^{2n}$, known bounds for the lengths of periodic Reeb orbits yield Corollary \ref{thmintro:el}. We prove both corollaries for completeness, even though the key points of the proofs are known facts in contact dynamics.

\subsection{Periodic Reeb orbits}

To prove Corollary  \ref{cor:multiplicity}, we need to show that the action bound given by the pinching condition excludes multiplicities in the absence of a short orbit. 
This follows from Lemma \ref{lem:Ham=Reeb} as follows:
Assume that one of the orbits $\gamma_{i}$ found in the theorem above corresponds to a Reeb orbit $z_{i}$ which is an $m$-fold cover, $m\geq 2$, of $z_1$ (the Reeb orbit corresponding to $\gamma_1$.)
According to Lemma \ref{lem:Ham=Reeb}, the corresponding periods are $h'(\bar{c}_{i})$ for $z_{i}$ and $h'(\bar{c}_1)$ for $z_1$, where the constants $\bar{c}_1$ and $\bar{c}_i$ are determined by Lemma \ref{lem:Ham=Reeb}. 
The orbit $z_{i}$ being an $m$-fold cover of $z_1$ translates into
\[
	h'(\bar{c}_{i})=mh'(\bar{c}_{1}).
\]
By construction of the function $h$, namely by second property in \eqref{eqn:properties_of_h}, we obtain 
\[
h'(\bar{c}_{1})<\frac{2}{m}\leq1.
\]
Thus we found a Reeb orbit of $\alpha_f$ with period less than 1.
To get the statement of Corollary \ref{cor:multiplicity}, recall that we normalized the radius by rescaling such that $\pi R_1^2=1$. Dropping this normalization gives the period bound in Corollary \ref{cor:multiplicity} and completes the proof.

%
%
%
%
%

\subsection{Starshaped hypersurfaces in $\R^{2n}$}

In this last part, we study the particular case of $\R^{2n}$, where we have concrete bounds for the length of Reeb orbits on starshaped hypersurfaces.
Note that there is a change in notation in this section to match the ``standard'' notation used for this theorem in the literature. In particular, we do not need to use the language of the line bundle over the symplectic manifold $M$ in this setting and for a given starshaped hypersurface, we do not use the defining function $f$ any more.

Therefore, $\Sigma$ will now denote the starshaped hypersurface of interest, which was denoted by $\Sigma_f$ above. Similarly, we now denote by $\alpha$ the usual contact form on $\Sigma$. In computations, we let $\alpha_x$ denote the form at the point $x\in\Sigma$.
We reprove below the relation between the largest radius $R_1$ of a sphere contained in $\Sigma$ and the action of periodic Reeb orbits on $\Sigma$. Together with Corollary \ref{cor:multiplicity}, this yields the desired Corollary \ref{thmintro:el}.


\begin{Lemma}\label{lemma:action}
Let $\gamma : [0,T] \rightarrow \Sigma$ be a simple $T$-periodic Reeb orbit on $\Sigma\subset\R^{2n}$ such that the largest sphere contained in the domain bounded by $\Sigma$ has radius $R_1$.
Assume moreover that for all $x \in \Sigma$, we have $T_{x}\Sigma \cap B_{R_{1}}(x_{0}) = \emptyset$.
Then we have 
\[
T \geq \pi R_{1}^{2}.
\]
\end{Lemma}
\begin{Rmk}
	The assumption that $T_{x}\Sigma \cap B_{R_{1}}(x_{0}) = \emptyset$ for all $x\in\Sigma$ can be reformulated as
	\begin{equation}\label{hyp:rmk}
		\langle \nu_{\Sigma}(z),z\rangle > R_{1} \quad \forall z \in \Sigma
	\end{equation}
	where $\nu_{\Sigma}(z)$ is the exterior normal vector of $\Sigma$ at point $z$ and $\langle\cdot,\cdot\rangle$ denotes the Euclidean scalar product on $\R^{2n}$.
	This condition is weaker than convexity which is also a common condition in similar settings.
\end{Rmk}
\begin{proof}
	We follow a similar argument in \cite{BLMR}.
	Let $\gamma : [0,T] \rightarrow \Sigma$ be a simple periodic Reeb orbit.
	We first compute a bound for $T$ in terms of the Reeb vector field. The main ingredient is the special form  of the contact form which is given as $\alpha_{x}(X_{x}) = \tfrac{1}{2} \langle X_{x}, Jx \rangle$. Writing $\bar{\gamma}(t) := \gamma(t) - \int_{0}^{T} \gamma(t) dt$, we compute
	\begin{align}
		2T & =2 \int_0^T \alpha_{\gamma(t)}\bigl(\dot{\gamma}(t)\bigr)dt  \nonumber\\
		&= \int_{0}^{T} \langle \dot{\gamma}(t), J\gamma(t)\rangle dt \qquad \nonumber\\
		& =  \int_{0}^{T} \langle \dot{\gamma}(t), J\bar{\gamma}(t)\rangle dt \nonumber\\
		& \leq  \| \dot{\gamma}\|_{L^{2}}  \| \bar{\gamma}\|_{L^{2}} \nonumber\\
		& \leq   \| \dot{\gamma}\|_{L^{2}}^{2} \tfrac{T}{2\pi}\nonumber\\
		& =  \frac{T}{2\pi}\int_{0}^{T} \|\dot{\gamma}(t)\|^{2} dt\nonumber \\
		& =  \frac{T}{2\pi}\int_{0}^{T} \|(R_{\alpha})_{\gamma(t)}\|^{2} dt,\label{borneaction}
	\end{align}
where we use Wirtinger's inequality to get the second inequality. For any point $x$ in $\Sigma$, the norm  of the Reeb vector field is bounded by $\|(R_{\alpha})_{x}\| \leq \frac{2}{R_{1}}$. Indeed,  we have
\begin{equation}
\iota(J\nu_{\Sigma})d\alpha(Y)=\omega(J\nu_{\Sigma},Y)=-\langle \nu_{\Sigma},Y\rangle=0
\end{equation}
	 for all $Y\in T\Sigma$.
Therefore, $R_{\alpha}$ is proportional to $J\nu_{\Sigma}$ and we have
$R_{\alpha}=cJ\nu_{\Sigma}$ for some function $c:\Sigma\to\R$. Since $\nu_\Sigma$ is the exterior normal and $J$ is an isometry we have $|c|=\|R_{\alpha}\|$.

On the other hand, we also use the second defining equation for the Reeb vector field to get 
\begin{equation}
1=\alpha_x({R_{\alpha}}_x)=\tfrac{1}{2}\langle c_xJ\nu_{\Sigma}(x),Jx\rangle=\frac{c_x}{2}\langle\nu_{\Sigma}(x),x\rangle
\end{equation}
and therefore, we find $c_x=\frac{2}{\langle\nu_{\Sigma}(x),x\rangle}\leq \frac{2}{R_{1}}$.

This gives rise to an upper bound for the last line in \eqref{borneaction}. Namely, we have
\begin{equation}
\frac{T}{2\pi}\int_{0}^{T} \|(R_{\alpha})_{\gamma(t)}\|^{2} dt\leq \frac{4}{R_{1}^{2}} T \frac{T}{2\pi}
\end{equation}
and in total, we have shown that $2T \leq 2T \frac{T}{\pi R_{1}^{2}}$ which implies the lemma.
\end{proof}

Finally, using this lemma, we can prove Corollary \ref{thmintro:el} to obtain the theorem by Ekeland-Lasry as a cuplength estimate.

\begin{proof}[Proof of Corollary \ref{thmintro:el}]
We now view the $2n-1$ sphere as the boundary of the ball blown-up at the origin. This point of view gives the sphere as a circle bundle in the tautological complex line bundle $\mathcal{O}(-1)$ over $\CP^{n-1}$.
	Note that the Reeb dynamics is unaffected by this consideration.
	Theorem \ref{Main Theorem} gives us the existence of $n$ periodic Reeb orbits on the sphere whose action is ``pinched''
	\[
		\pi R_1^2<\A(\gamma_1)<\hdots<\A(\gamma_{n})<\pi R_2^2.
	\]
	The condition $R_2^2<2R_1^2$ corresponds to the above condition that $R_0<2$ and therefore, these $n$ Reeb orbits cannot be iterates of one another.
The lower bound on the period of periodic Reeb orbits on $\Sigma$ given by Lemma \ref{lemma:action} above shows that they can also not be iterates of a short orbit.
	Thus we have $n$ simple periodic Reeb orbit.
\end{proof}


%
\bibliographystyle{amsalpha}
\bibliography{bibtex_paper_list}

\end{document}